\documentclass[journal]{IEEEtran}
\usepackage[pagebackref]{hyperref}
\usepackage{xcolor}
\usepackage{booktabs}
\usepackage[sort,compress,numbers]{natbib}
\usepackage{makecell}
\usepackage{subcaption}
\usepackage{enumitem}

\hypersetup{
    colorlinks,
    linkcolor={red!50!black},
    citecolor={blue!50!black},
    urlcolor={blue!80!black}
}

\usepackage{cuted}

\usepackage{xspace}
\usepackage{bbm}
\usepackage{bm}
\input{mathlig}
\usepackage{mathtools}
\usepackage{relsize}
\usepackage{mathrsfs}
\usepackage{dsfont}
\DeclarePairedDelimiterX{\inp}[2]{\langle}{\rangle}{#1, #2}
\makeatletter
\newcommand*\bigcdot{\mathpalette\bigcdot@{.5}}
\newcommand*\bigcdot@[2]{\mathbin{\vcenter{\hbox{\scalebox{#2}{$\m@th#1\bullet$}}}}}
\makeatother

\newcommand{\muspace}{\mspace{1mu}}

\DeclareRobustCommand{\scond}{\mathchoice{\muspace\vert\muspace}{\vert}{\vert}{\vert}}
\mathlig{|}{\scond}

\DeclareRobustCommand{\discint}{\mathchoice{\mspace{-1.5mu}:\mspace{-1.5mu}}{\mspace{-1.5mu}:\mspace{-1.5mu}}{:}{:}}
\mathlig{::}{\discint}
\newcommand{\suchthat}{\mathchoice{\colon}{\colon}{:\mspace{1mu}}{:}}

\newcommand{\Cc}{\mathcal{C}}

\newcommand{\Mc}{\mathcal{M}}

\newcommand{\Xc}{\mathcal{X}}

\newcommand{\xv}{{\bf x}}

\newcommand{\cb}{{\mathbf c}}

\newcommand{\eb}{{\mathbf e}}

\newcommand{\xb}{{\mathbf x}}

\newcommand{\rhob}{\bm{\rho}}

\def\a{\alpha}
\def\b{\beta}

\def\d{\delta}

\def\th{\theta}
\def\Th{\Theta}

\DeclareMathOperator\E{\mathsf{E}}
\let\P\relax
\DeclareMathOperator\P{\mathsf{P}}

\newcommand\eg{e.g.,\xspace}
\newcommand\ie{i.e.,\xspace}
\def\textiid{i.i.d.\@\xspace}
\newcommand\iid{\ifmmode\text{ i.i.d. } \else \textiid \fi}

\newcommand{\Real}{\mathbb{R}}
\newcommand{\Natural}{\mathbb{N}}

\newcommand{\ones}{\mathds{1}}

\newcommand{\half}{\frac{1}{2}}%

\def\mathllap{\mathpalette\mathllapinternal}
\def\mathllapinternal#1#2{%
  \llap{$\mathsurround=0pt#1{#2}$}}

\def\clap#1{\hbox to 0pt{\hss#1\hss}}
\def\mathclap{\mathpalette\mathclapinternal}
\def\mathclapinternal#1#2{%
  \clap{$\mathsurround=0pt#1{#2}$}}

\let\oldstackrel\stackrel
\renewcommand{\stackrel}[2]{\oldstackrel{\mathclap{#1}}{#2}}

\DeclarePairedDelimiterX{\infdivx}[2]{(}{)}{%
  #1\;\delimsize\|\;#2%
}

\renewcommand{\hbar}{h\mathllap{\overline{\vphantom{h}\hphantom{\rule{4.6pt}{0pt}}}\mspace{0.77mu}}}

\catcode`~=11 %
\newcommand{\urltilde}{\kern -.06em\lower -.06em\hbox{~}\kern .02em}
\catcode`~=13 %

\DeclareMathOperator*{\argmax}{arg\,max}

\DeclarePairedDelimiterX{\norm}[1]{\lVert}{\rVert}{#1}
\DeclarePairedDelimiterX{\abs}[1]{\lvert}{\rvert}{#1}

\usepackage{xparse}

\newcommand*\diff{\mathop{}\!\mathrm{d}}

\let\oldpartial\partial
\renewcommand*{\partial}{\mathop{}\!\oldpartial}

\newcommand{\defeq}{\mathrel{\mathop{:}}=}

\newcommand{\wealth}{\mathsf{Wealth}}

\newcommand{\qtkt}{\qt^{\mathsf{KT}}}

\newcommand{\wealthup}{\wealth^{\texttt{UP}}}

\newcommand\numberthis{\addtocounter{equation}{1}\tag{\theequation}}

\usepackage{amsmath,amsthm,amssymb}
\newtheorem{theorem}{Theorem}

\newtheorem{lemma}[theorem]{Lemma}
\newtheorem{corollary}[theorem]{Corollary}

\newtheorem{proposition}[theorem]{Proposition}
\theoremstyle{definition}

\newtheorem{remark}[theorem]{Remark}

\usepackage{thmtools}
\declaretheoremstyle[
  headfont=\color{red}\normalfont\bfseries,
  bodyfont=\color{red}\normalfont\itshape,
]{colored}
\declaretheoremstyle[
  headfont=\color{blue}\normalfont\bfseries,
  bodyfont=\color{blue}\normalfont\itshape,
]{resolved}

\declaretheoremstyle[
  headfont=\color{blue}\normalfont\bfseries,
  bodyfont=\color{blue}\normalfont\itshape,
]{blue}

\declaretheoremstyle[
  headfont=\color{red}\normalfont\bfseries,
  bodyfont=\color{red}\normalfont\itshape,
]{red}

\renewcommand{\E}{\mathop{\mathsf{E}}}

\newcommand{\qt}{\tilde{q}}

\newcommand{\pv}{\mathbf{p}}

\renewcommand{\epsilon}{\varepsilon}
\renewcommand{\tilde}{\widetilde}
\renewcommand{\hat}{\widehat}
\renewcommand{\th}{\theta}

\newlength{\depthofsumsign}
\makeatletter
\newcommand*\dotp{\mathpalette\dotp@{.5}}
\newcommand*\dotp@[2]{\mathbin{\vcenter{\hbox{\scalebox{#2}{$\m@th#1\bullet$}}}}}
\makeatother

\newcommand{\bb}{\mathbf{b}}

\newcommand{\ct}{\tilde{c}}
\newcommand{\low}{\mathrm{low}}
\newcommand{\up}{\mathrm{up}}

\renewcommand{\rhob}{\bm{\rho}}

\makeatletter
\newcommand\footnoteref[1]{\protected@xdef\@thefnmark{\ref{#1}}\@footnotemark}
\makeatother

\newif\iftit
\renewcommand{\defeq}{\triangleq}

\tittrue

\newcommand{\scndrevision}[1]{\textcolor{black}{#1}}
\newcommand{\eqnref}[1]{Eq.~\eqref{#1}}

\iftit
\newcommand{\titrevision}[1]{\textcolor{black}{#1}}
\else
\newcommand{\titrevision}[1]{\textcolor{black}{#1}}
\fi

\begin{document}
%
\title{On Confidence Sequences for Bounded Random Processes via Universal Gambling Strategies}
%
%
%

\author{J. Jon Ryu
and~Alankrita~Bhatt
\thanks{Manuscript received June 27, 2023; revised August 13, 2024; accepted August 20, 2024. 
This work was supported in part by the National Science Foundation under Grant CCF-1911238.
\emph{(This work was done in part while the authors were graduate students at UC San Diego.)
(Corresponding author: J.~Jon~Ryu.)}}
\thanks{J. Jon Ryu is with the Department of Electrical Engineering and Computer Science, Massachusetts Institute of Technology, Cambridge,
MA 02139, USA. e-mail: \texttt{\href{mailto:jongha@mit.edu}{jongha@mit.edu}}.}
\thanks{Alankrita Bhatt is with the Computing + Mathematical Sciences Department, California Institute of Technology, Pasadena, CA 91125, USA. e-mail: \texttt{\href{mailto:abhatt@caltech.edu}{abhatt@caltech.edu}}.}
\thanks{Communicated by F.~Orabona, Associate Editor for Machine Learning and Statistics.}}
\maketitle

\begin{abstract}
This paper considers the problem of constructing a confidence sequence, which is a sequence of confidence intervals that hold uniformly over time, for estimating the mean of bounded real-valued random processes.
This paper revisits the gambling-based approach established in the recent literature from a natural \emph{two-horse race} perspective, and demonstrates new properties of the resulting algorithm induced by Cover (1991)'s universal portfolio.
The main result of this paper is a new algorithm based on a mixture of lower bounds, which closely approximates the performance of Cover's universal portfolio with constant per-round time complexity. 
A higher-order generalization of a lower bound on a logarithmic function in (Fan et al., 2015), which is developed as a key technique for the proposed algorithm, may be of independent interest.

\end{abstract}


%
\IEEEpeerreviewmaketitle

\section{Introduction}
\iftit
\IEEEPARstart{S}{uppose}
\else
Suppose 
\fi
that $(Y_t)_{t=1}^{\infty}$ is a $[0,1]$-valued stochastic process such that $\E[Y_t|Y^{t-1}]\equiv \mu$ for any $t\ge 1$, for some unknown mean parameter $\mu\in(0,1)$. 
Here we use a shorthand notation $Y^t\defeq (Y_1,\ldots, Y_t)$.
A \emph{confidence sequence} for the mean parameter $\mu$ at level $1-\d$ is defined as a sequence of sets $(\Cc_t)_{t=1}^{\infty}$ such that $\Cc_t\subseteq (0,1)$ is a function of $Y^{t-1}$ and
\[
\P(\mu\in \Cc_t, \forall t\ge 1)\ge 1-\d.
\]
A confidence sequence under this setting can be applied to solving some fundamental problems in statistics, such as sequentially estimating the mean of a bounded distribution with \iid samples~\citep{Waudby-Smith--Ramdas2020b,Orabona--Jun2021}, or constructing a time-uniform confidence interval for the mean of a fixed set of numbers by random sampling without replacement~\citep{Waudby-Smith--Ramdas2020a}.
{This notion is better suited in real applications than the classical confidence sets (or intervals) which only applies to a single time instance.
For example, suppose that a practitioner wishes to sequentially estimate the mean of such a stochastic process. 
The user wishes to construct a confidence interval, so that she can choose when to stop sample to achieve a desired precision and confidence level a priori. 
A confidence sequence allows the user to determine adaptively in the sequential inference setting thanks to the time uniformity; see \citep{Ramdas--Ruf--Larsson--Koolen2020} and \citep{Grunwald--de-Heide--Koolen2020}.}

The idea of time-uniform confidence sequences dates back to \citet{Hoeffding1963}, \citet{Darling--Robbins1967}, and \citet{Lai1976}.
\titrevision{Rather ignored for past decades}, 
the idea has been revived recently in the statistics community by a series of papers~\citep{Ramdas--Ruf--Larsson--Koolen2020,Waudby-Smith--Ramdas2020b,Howard--Ramdas--McAuliffe--Sekhon2021} and in the computer science community~\citep{Jun--Orabona2019,Orabona--Jun2021}, to name a few, even beyond the boundedness assumptions on the stochastic processes we consider in this paper. 

As noted above, this paper studies how to construct a confidence sequence for the mean of bounded random processes (with known support), especially via the ``gambling'' approach established by a recent line of research. 
The most closely related work is \citep{Waudby-Smith--Ramdas2020b,Orabona--Jun2021}.
\titrevision{\citet{Waudby-Smith--Ramdas2020b} proposed a gambling framework to construct confidence sequences of bounded real-valued sequences.
Several betting strategies were proposed, and their analytical behavior of resulting confidence sequences were established.
\citet{Orabona--Jun2021} proposed to apply the universal portfolio (UP) \titrevision{method} of \citet{Cover1991} to construct a confidence sequence, and obtained an analytical expression that bounds the UP-induced confidence sequence based on \titrevision{the regret upper bound (which corresponds to a wealth \emph{lower} bound)} of UP in terms of the logarithmic wealth.
They demonstrated an excellent empirical performance of the proposed algorithms especially in a small-sample regime.}
The proposed algorithms dubbed as PRECiSE, however, is based on a rather intricate regret analysis and does not naturally result in confidence \emph{intervals}. 
\titrevision{In their work, \citet{Orabona--Jun2021} also established a more relaxed lower bound that directly generates confidence intervals without introducing excessive slackness.}

\titrevision{The main goal of this paper is twofold.
First, we provide a new perspective of the \emph{two-horse race} to the existing gambling-based approach which was proven effective by \citet{Waudby-Smith--Ramdas2020a,Waudby-Smith--Ramdas2020b,Orabona--Jun2021}, which admits a more natural interpretation in gambling and leads to a conceptual simplification.
Building upon the framework, we propose a new method to approximate the performance of the UP closely and fast in a more direct manner than \citet{Orabona--Jun2021}, without invoking the regret analysis of UP.
Our contributions can be summarized as follows.}

\begin{enumerate}

\item For $\{0,1\}$-valued sequences, we show that Cover's UP can be simplified to a universal two-horse race scheme, which can be implemented with constant complexity per round to construct time-uniform confidence intervals. 
This observation is simple, but has not been realized in the literature.

\item For $[0,1]$-valued sequences, we discuss how to exactly compute the wealth of any mixture of wealth processes of constant bettors (such as Cover's UP) with $O(t)$ complexity at round $t$ and show that it leads to time-uniform confidence intervals.
To circumvent the cumbersome $O(t)$ per-round complexity of the UP, we propose a new UP-like method based on a new lower bound of the wealth of Cover's UP with constant complexity per round. 
{Our algorithm is based on Lemma~\ref{lem:generalized_lower_bound}, a higher-order generalization of the proof technique of \citet[Lemma~4.1]{Fan--Grama--Liu2015} developed for proving exponential inequalities for martingales, which may be of independent interest. Experiments validate that the proposed algorithm gives tight confidence sequences.}
\end{enumerate}

The rest of the paper is organized as follows.
Section~\ref{sec:prelim} is devoted to mathematical preliminaries.
We start with a special case of $\{0,1\}$-valued processes in Section~\ref{sec:hr} and show that there is a simple two-horse-race-based algorithm that emulates the performance of Cover's UP.
In Section~\ref{sec:up}, we then study the general case of $[0,1]$-valued processes and propose a new mixture-of-lower-bound approach.
Related work are briefly discussed in Section~\ref{sec:related} and 
experimental results are presented in Section~\ref{sec:exps}.
We conclude the paper in Section~\ref{sec:conclusion} with some remarks.
All the deferred proofs and technical lemmas can be found in Appendix.

\section{Preliminaries}
\label{sec:prelim}

At its core, the common technique upon which most, if not all, methods for constructing confidence sequences rely on is the following celebrated inequality by \citet{Ville1939} from the martingale theory.
\begin{theorem}[Ville's inequality]
\label{thm:ville}
For a nonnegative supermartingale sequence $(W_t)_{t=0}^{\infty}$ with $W_0>0$, for any $\d>0$, we have
\[
\P\Bigl\{\sup_{t\ge1} \frac{W_t}{W_0} \ge \frac{1}{\d}\Bigr\}\le \d.
\]
\end{theorem}
This statement is a sequential, uniform counterpart for nonnegative supermartingales to Markov's inequality for nonnegative random variables.
Since this inequality controls the probability that the sequence $(W_t)_{t=0}^{\infty}$ overshoots a certain threshold uniformly over time, 
the key idea is to construct a \emph{good} martingale sequence $(W_t)_{t=0}^{\infty}$ out of the given sequence $(Y_t)_{t=1}^{\infty}$ \scndrevision{such that the wealth sequence grows as rapidly as possible for any $m\neq \mu$,}
so that the sequence of events controlled by this inequality can be translated to a \scndrevision{tighter} confidence sequence for a target parameter $\mu$; \scndrevision{the intuition will be visually demonstrated in Fig.~\ref{fig:evolution} in Section~\ref{sec:exp:evolution}}.

\titrevision{As explored in \citep{Hendriks2021,Jun--Orabona2019,Waudby-Smith--Ramdas2020b,Orabona--Jun2021}}, a natural way to construct a supermartingale is via a gambling, and we thus start by introducing the gambling formalism of this approach.

\subsection{Supermartingales from Gambling}
A natural way to construct a (super)martingale sequence is to consider the wealth process from a \emph{(sub)fair gambling}.
Let $K\ge 2$ and 
let $\Mc\subseteq\Real_+^K$ be a set of \emph{odds vectors}.
We consider an abstract setting of gambling defined by the following \emph{multiplicative} game.
Let $\wealth_0>0$ be the initial wealth of a gambler.
For each time $t\ge 1$, a gambler chooses a bet $\bb_t=\bb_t(\xb^{t-1})\in\Delta_{K-1}$ as a function of the previous odds vectors $\xb^{t-1}$; such a betting strategy is said to be \emph{causal} (or \emph{nonanticiapting}). 
\titrevision{Here, $\Delta_{K-1}\defeq \{\pv\in\Real_{\ge0}^K \suchthat p_1+\ldots+p_K=1\}$ denotes the $(K-1)$-dimensional probability simplex.}
At the end of each round, the odds vector $\xb_t$ is revealed, and the gambler's wealth gets \emph{multiplied} by $\langle\bb_t,\xb_t\rangle$.
Therefore, after round $t$, the gambler's wealth can be written as 
\[
\wealth_t(\xb^t) = \wealth_0 \prod_{i=1}^t \langle \bb_i,\xb_i\rangle.
\]

We call the odds vector sequence $(\xb_t)_{t=1}^{\infty}$ \emph{subfair}, if $\E[\xv_t|\xv^{t-1}]\le\ones$ for every $t$, where the inequality holds coordinatewise \titrevision{and $\ones\defeq (1,\ldots,1)$ denotes the all-one vector.}
Under this condition, the resulting wealth process from this game is always a supermartingale as one may intuitively expect.
\begin{proposition}
\label{prop:wealth_is_martingale}
In any gambling with the cumulative wealth of the form
$\wealth_t = \wealth_0 \prod_{i=1}^t \langle \bb_i,\xb_i\rangle$,
if $(\xb_t)_{t=1}^{\infty}$ is (sub)fair, then any wealth process $(\wealth_t)_{t=1}^{\infty}$ attained by a causal betting strategy is a (super)martingale.
\end{proposition}
\begin{proof}
For every $t$, we have
\begin{align*}
\E[\wealth_t|\xv^{t-1}]
&=\wealth_{t-1}\langle \bb_t,\E[\xv_t|\xv^{t-1}]\rangle \\
&\le \wealth_{t-1} \langle \bb_t,\ones\rangle =\wealth_{t-1}.\qedhere
\end{align*}
\end{proof}

{We remark that this abstract formulation captures several standard settings of gambling as special instances; see Table~\ref{tab:gambling_examples}. 
For example, if $\Mc=\{2\eb_1,2\eb_2\}$, it corresponds to \titrevision{the standard (fair) \emph{coin toss}~\citep{Cover--Thomas2006} (or coin betting~\citep{Orabona--Pal2016}) game.}
If $\Mc=\{o_1\eb_1,\ldots,o_K\eb_K\}$ for some $o_1,\ldots,o_K>0$, then the game becomes a \emph{horse race} with odds $o_1,\ldots,o_K$~\citep[Chap.~6]{Cover--Thomas2006}.
The most general case is when $\Mc=\Real_{>0}^K$, where the game is typically known as a \titrevision{\emph{portfolio selection}} in a stock market~\citep{Cover1991}.
The \emph{continuous} coin toss problem that corresponds to $\Mc=\{[2\ct,2(1-\ct)]\suchthat \ct\in[0,1]\}$ was considered by \citet{Orabona--Pal2016} to translate the idea of universal betting for coin toss to the domain of online linear optimization. In what follows, betting strategy and portfolio are used interchangeably, where the latter is used typically when the outcomes are continuous like the stock market.}

\subsection{Two-Horse Race and Its Continuous Extension}
In this paper, we consider the \emph{two-horse race} with odds $o_1,o_2>0$, which corresponds to $\Mc_{o_1,o_2}\defeq\{c o_1\eb_1+(1-c)o_2\eb_2\suchthat c\in\{0,1\}\}$. 
Each odds vector at time $t$ can be written as $\xb_t=[o_1c_t,o_2(1-c_t)]$ with $c_t\in\{0,1\}$, where $c_t$ can be understood as an index of the winning horse. 
\titrevision{Given the fixed odds $o_1$ and $o_2$, the odds vector sequence $x_t$ is fully characterized by the outcomes $c_t$ so we will often refer to this sequence $c_t$ instead of the full odds vector $\xb_t$ in our discussion below.}
Since $c_i\in\{0,1\}$, the multiplicative gain at each round is equivalently expressed as
\[
o_1c_i b_i + o_2(1-c_i)(1-b_i) 
= (o_1 b_i)^{c_i} (o_2(1-b_i))^{1-c_i},
\]
where $b_i$ denotes the fraction of bet on horse 1.

As a natural extension, we also consider the \emph{continuous} two-horse race game with odds $o_1,o_2>0$ which corresponds to $\tilde{\Mc}_{o_1,o_2}\defeq\{\ct o_1\eb_1+(1-\ct)o_2\eb_2\suchthat \ct\in[0,1]\}$. Now, unlike the previous case where there is a single winner for each round, the outcome of the game $\ct$ is $[0,1]$-valued.
Strictly speaking, this game can be better understood as a \emph{structured} two-stock market, as there is no single winning horse in this setup.

\titrevision{We note that we use $c_t\in\{0,1\}$ and $\ct_t\in[0,1]$ to denote generic outcomes which could be deterministic, as opposed to the stochastic process $(Y_t)_{t=1}^\infty$.}

\subsection{The Blueprint}
\label{sec:blueprint}
We now describe the gambling formalism we take in this paper, which we adopt from \citep{Hendriks2021,Jun--Orabona2019,Waudby-Smith--Ramdas2020b,Orabona--Jun2021}. 
Most of the following exposition can be also found in \citep[Theorem~1]{Waudby-Smith--Ramdas2020b} and \citep[Section~4.1]{Waudby-Smith--Ramdas2020b}, but we also include \titrevision{some new arguments such as the two-horse-race-based exposition and the converse of Corollary~\ref{cor:prop:wealth_is_martingale} stated below.}

Recall that we sequentially observe a sequence $(Y_t)_{t=1}^{\infty}$ such that $\E[Y_t|Y^{t-1}]\equiv \mu$ for every $t\ge 1$ and some $\mu\in(0,1)$, and the goal is to estimate $\mu$ at each time $t$ based on the observation $Y^{t-1}$.
To estimate the unknown $\mu$, 
we consider the two-horse race with the odds vector 
\[
\xb_t=\xb_t(Y_t;m)=\Bigl[\frac{Y_t}{m},\frac{1-Y_t}{1-m}\Bigr]
\numberthis\label{eq:odds_vector}
\]
for a parameter $m\in(0,1)$. 
Then, as a corollary of Proposition~\ref{prop:wealth_is_martingale}, we can prove:

\begin{corollary}
\label{cor:prop:wealth_is_martingale}
Assume that $\E[Y_t|Y^{t-1}]\equiv \mu$ for every $t\ge 1$ and some $\mu\in(0,1)$. 
Then, the wealth process from the continuous two-horse race 
with the odds vector $\xv_t= [\frac{Y_t}{m},\frac{1-Y_t}{1-m}]$ is a martingale if $m=\mu$.
Conversely, if $m\neq \mu$, there exists a causal betting strategy whose wealth process is strictly submartingale.
\end{corollary}

\begin{proof}
If $m=\mu$, $(\xb_t)_{t=1}^{\infty}$ is fair by construction, \ie $\E[\xb_t|\xb_{t-1}]=[\frac{\mu}{m},\frac{1-\mu}{1-m}]=[1,1]$, so we can apply Proposition~\ref{prop:wealth_is_martingale}.
For the case $m\neq \mu$, we assume $m<\mu$ without loss of generality; the proof for $m>\mu$ is similar. 
Then, $b\mapsto\frac{\mu}{m}b + \frac{1-\mu}{1-m}(1-b)$ becomes a strictly increasing function, and 
\[
\frac{\mu}{m}b + \frac{1-\mu}{1-m}(1-b) > 1
\]
for any $b\in(m,1)$.
Hence, in this case, any constant betting strategy $\bb_t=[b,1-b]$ with $b\in(m,1)$ satisfies
\begin{align*}
\E[\wealth_t|\xv^{t-1}]
&=\wealth_{t-1}\langle \bb_t,\E[\xv_t|\xv^{t-1}]\rangle\\
&=\wealth_{t-1} \Bigl(\frac{\mu}{m}b + \frac{1-\mu}{1-m}(1-b)\Bigr)\\
&> \wealth_{t-1},
\end{align*}
implying that the wealth process is strictly submartingale.
\end{proof}

Suppose that we have a gambling strategy $\bb_t(\xb^{t-1};m)$ for each $m$, and let $\wealth_t(\xb^{t};m)$ be the wealth process of the strategy for the two-horse race game in~\eqnref{eq:odds_vector} parameterized by $m$.\footnote{All the gambling strategies $\bb_t(\xb^{t-1};m)$ considered in the current paper are independent of $m$, but we explicitly include the dependence to subsume a general use-case. For example, \citet{Waudby-Smith--Ramdas2020a} studied betting strategies that depend on $m$.}
Since $(\wealth_t(\xv^t;\mu))_{t\ge 0}$ is a martingale (Corollary~\ref{cor:prop:wealth_is_martingale}),
by Ville's inequality (Theorem~\ref{thm:ville}), for any $\d \in(0,1)$, 
\[
\P\Bigl\{\sup_{t\ge1} \frac{\wealth_t(\xb^{t};\mu)}{\wealth_0}\ge \frac{1}{\delta}\Bigr\} \le \delta,
\numberthis\label{eq:apply_ville}
\]
or equivalently,
\[
\P\{\mu \in \Cc_t(Y^t;\d), \forall t\ge 1\} \ge 1-\delta,
\]
where we define the confidence set 
\[
\Cc_t(Y^t;\d)
\defeq\Bigl\{m\in(0,1)\suchthat \sup_{1\le i\le t} \frac{\wealth_i(\xb^i;m)}{\wealth_0}
< \frac{1}{\d}\Bigr\}.
\]
Since $(\Cc_t(Y^t;\d))_{t=1}^{\infty}$ is nonincreasing (\ie $C_{t-1}\subseteq \Cc_t$) by construction, we can interpret that the confidence sequence $(\Cc_t)_{t=1}^{\infty}$ sequentially \emph{refines} the estimate for $\mu$ at every round $t$ by \emph{rejecting} a candidate parameter $m$ such that 
$\wealth_t(\xb^t;m)>\frac{1}{\d}$.

Here is a high-level interpretation of the gambling approach.
Hypothetically, we run the two-horse race with parameter $m$ for each $m\in (0,1)$ based on the observation $(Y_t)_{t=1}^{\infty}$.
At the very beginning, we start with the entire interval $C_0=(0,1)$ as the candidate list for $\mu$.
After each round of gambling upon observing $Y_t$, we compute the cumulative wealth from each horse race, and if it exceeds a prescribed threshold, which is $1/\d$ for a level-$(1-\d)$ confidence sequence, we remove the associated parameter $m$ from the candidate list $C_{t-1}$; at the end of the round, $\Cc_t$ then consists of the remaining values.
Intuitively, we can remove such values $m$ with confidence as the wealth process from the horse-race with parameter $\mu$ is martingale, and thus will not exceed the threshold most likely; that is, if the wealth exceeds the threshold, the parameter $m$ is likely not $\mu$.

This intuition explains why we can expect a better gambling strategy to result in a tighter confidence sequence:
a \emph{good} strategy would grow its wealth as fast as possible from the horse races with $m\neq \mu$, and thus rejects those values of $m$ at an early stage with few observations.
Note that the converse part of Corollary~\ref{cor:prop:wealth_is_martingale} ensures that there exists a betting strategy that makes the wealth process a strict submartingale. 
In an ideal scenario, therefore, one can expect that a good gambling strategy may make arbitrarily large money in the long run from the horse race with $m\neq \mu$, leaving only the true $\mu$ in the candidate list.

\begin{remark}[Different gambling parameterization]
While the setting of gambling we consider is equivalent to those in \citep{Waudby-Smith--Ramdas2020b,Orabona--Jun2021}, the literature uses a different convention.
For the two-horse race setting with odds $\frac{1}{m}$ and $\frac{1}{1-m}$ and a betting strategy $(b_t)_{t=1}^{\infty}$, they write the multiplicative wealth as
\[
\frac{1}{m} \ct_t b_t + \frac{1}{1-m} (1-\ct_t)(1-b_t)
=1+\lambda_t(m) (\ct_t-m),
\]
by viewing the single number $\ct_t-m\in[-m, 1-m]$ as an outcome of the horse race and defining
\[
\lambda_t(m) \defeq \frac{b_t}{m(1-m)}-\frac{1}{1-m}\in \Bigl[-\frac{1}{1-m}, \frac{1}{m}\Bigr]
\numberthis
\label{eq:signed_betting}
\]
as a \titrevision{\emph{scaled bet}}.
The downside of this standard convention is that the (scaled) betting in~\eqnref{eq:signed_betting} must depend on the underlying odds (and thus on the parameter $m$) by the range it can take, \titrevision{which we view as rather unnatural.}
We believe that
the horse-race language we adopt in this paper separates $m$ from the betting strategy, and thus admits a cleaner interpretation.
\end{remark}

\begin{table*}[t]
\centering
\caption{Types of gambling. 
{The horse race with $K=2$ and $o_1=o_2=2$ reduces to the fair-coin toss. 
The stock investment becomes the horse race if the market vectors are restricted to scaled Kelly market vectors, \ie $\xb_t\in\{o_1\eb_1,\ldots,o_K\eb_K\}$, and becomes continuous horse race if the market vectors $\xb_t$ lie on the scaled simplex $(o_1,\ldots,o_K)\odot \Delta_{K-1}$, where $\odot$ denotes the elementwise product.
Note the following difference in the convention: in the coin toss and horse race, the outcomes are not associated with the odds, while the outcome of the stock investment, which is the market vector $\xb_t$, inherently captures the ``odds'' of the game.}}
\begin{small}
\begin{tabular}{c c c c}
\toprule
     & \textbf{Outcome} & \textbf{Bet} & \textbf{Multiplicative gain} $\frac{\wealth_t}{\wealth_{t-1}}$\\
\midrule
    Coin toss & $c_t\in\{0,1\}$ & $b_t\in[0,1]$ & $(2b_t)^{\ones\{c_t=1\}}(2(1-b_t))^{\ones\{c_t=0\}}$ \\
    Horse race & $\cb_t\in\{\eb_1,\ldots,\eb_K\}$ & $\bb_t\in\Delta_{K-1}$ & $(o_1b_{t1})^{\ones\{\cb_t=\eb_1\}}\cdots(o_Kb_{tK})^{\ones\{\cb_t=\eb_K\}}$ \\
\midrule
    Continuous coin toss & $\ct_t\in[0,1]$ & $b_t\in[0,1]$ & $2b_t \ct_t+ 2(1-b_t)(1-\ct_t)$ \\
    Continuous horse race & $\tilde{\cb}_t\in \Delta_{K-1}$ & $\bb_t\in\Delta_{K-1}$ & $o_1b_{t1}\ct_{t1}+\ldots+o_Kb_{tK}\ct_{tK}$ \\
\midrule
    Stock investment 
    & $\xb_t\in\Real_{>0}^K$
    & $\bb_t\in\Delta_{K-1}$ 
    & $\langle \xb_t, \bb_t\rangle$
    \\
\bottomrule     
\end{tabular}
\end{small}
\label{tab:gambling_examples}
\end{table*}

\subsection{Achievable Wealth Processes and the \scndrevision{Method of Mixtures}}
As described above, it suffices to directly construct a valid wealth process for a confidence sequence, without explicitly finding a causal betting strategy that achieves it.
In this paper, the central technique we use is to consider a wealth process defined by a \emph{mixture} of a collection of wealth processes.
The purpose of this section is to ensure that we can use any \titrevision{\emph{mixture of wealth processes} (or \emph{mixture wealth}} in short) in constructing a confidence sequence.

To show the main result of this section, Corollary~\ref{cor:mixture_is_achievable}, we first state the following theorem which characterizes when a given sequence of nonnegative functions can be \emph{realized} by a causal betting strategy, under a mild regularity assumption on the set of market vectors $\Mc$ that holds for all games in Table~\ref{tab:gambling_examples}.
\titrevision{In the statement below, we use $\circ$ to denote a concatenation between two sequences, \eg $\xv^{t-1}\circ \xv_t = \xv^{t}$.}
\titrevision{
\begin{theorem}
\label{thm:equivalent_condition_achievable_wealth}
Suppose that 
$\{o_1\eb_1,\ldots,o_K\eb_K\}\subseteq \Mc$ for some $o_1,\ldots,o_K>0$.
Let $(\Psi_t\suchthat \Mc^{t-1}\to \Real_{\ge 0})_{t=1}^{\infty}$ be a sequence of nonnegative functions, where each $\Psi_t$ at time $t$ maps past odds vectors $\xv^{t-1}\in\Mc^{t-1}$ to a nonnegative real number.
Assume that the following conditions hold:
\begin{itemize}
\item [(A1)] (Consistency)
For any $t\ge 1$, any $j\in[K]$, and any $\xb^{t-1}\in\Mc^{t-1}$,
\[
\Psi_{t-1}(\xb^{t-1})\ge \sum_{j=1}^K \frac{1}{o_j} \titrevision{\Psi_t(\xb^{t-1}\circ o_j\eb_j)}.
\numberthis
\label{eq:wealth_consistency}
\]
\item [(A2)] (Convexity)
For any $t\ge 1$, any $j\in[K]$, and any $\xb^{t}\in\Mc^{t}$, 
\[
\sum_{j=1}^K \frac{x_{tj}}{o_j} \titrevision{\Psi_t(\xb^{t-1}\circ o_j\eb_j)} \ge \Psi_t(\xb^t).
\numberthis
\label{eq:wealth_convexity}
\]
\end{itemize}
Then, there exists a causal betting strategy $\bb_t(\xb^{t-1})\in \Delta_{K-1}$ that guarantees wealth at least $\wealth_0\Psi_t(\xb^t)$ for any market sequence $\xb^t\in \Mc^t$.
Conversely, if $\wealth_0\Psi_t(\xb^t)$ is the wealth achieved by a causal betting strategy given $\xb^t$, then $(\Psi_t)_{t=1}^{\infty}$ satisfies (A1) and (A2) with equality.
\end{theorem}}
\titrevision{It is worth noting that \citet{Cover1966} studied an equivalent condition for achievable scores in predicting binary sequences under the Hamming score.
Theorem~\ref{thm:equivalent_condition_achievable_wealth} can be viewed as an analogous result for gambling, \scndrevision{which can be viewed as a prediction game under the log score.} 
Its proof, which can be found in Appendix~\ref{app:proofs}, is based on a simple inductive argument, but it is not a direct adaptation of Cover's proof.}

{\titrevision{Note that given any achievable wealth process, we can explicitly define a causal betting strategy that achieves the wealth process using in the proof of Theorem~\ref{thm:equivalent_condition_achievable_wealth} in Appendix~\ref{app:proofs}; see \eqnref{eq:wealth_to_betting} therein for an explicit expression of the betting strategy.}
{This implies that computing the cumulative wealth of a betting strategy and computing the betting strategy itself are essentially equivalent.}}
Another immediate consequence of this theorem is that, for a class of sequences of achievable wealth functions, their mixture is always achievable. 
\begin{corollary}
\label{cor:mixture_is_achievable}
Consider a set of betting strategies $(\bb_{t}^{\th})$, where for each strategy parameterized by $\th\in\Th$, each of whose wealth is lower bounded by $\wealth_0\Psi_{t}^{\th}(\xb^t)$ for any $t\ge 1$ and $\xb^t\in \Mc^t$.
Pick any probability measure $\pi(\th)$ on $\Th$ and define a mixture $\tilde{\Psi}_t^{\pi}(\xb^t)\defeq \int \Psi_{t}^{\th}(\xb^t)\pi(\diff\th)$. 
Then, there always exists a causal betting strategy whose wealth is lower bounded by $\wealth_0\tilde{\Psi}_t^{\pi}(\xb^t)$.
\end{corollary}
\begin{proof}
For each $\th\in\Th$, let $\wealth_t^{\th}(\xb^t)$ denote the wealth achieved by $(\bb_t^{\th})_{t=1}^{\infty}$ for $\xb^t$.
Since $(\wealth_t^{\th})_{t\ge 0}$ satisfies (A1) and (A2) automatically by the converse part of Theorem~\ref{thm:equivalent_condition_achievable_wealth}, so does the \emph{mixture wealth} defined by 
$\wealth_t^{\pi}(\xb^t)\defeq\int\wealth_t^{\th}(\xb^t)\pi(\diff\th)$
by linearity of expectation.
Therefore, by Theorem~\ref{thm:equivalent_condition_achievable_wealth}, there always exists a causal strategy whose wealth is at least $ \wealth_t^{\pi}(\xb^t) \ge \wealth_0 \tilde{\Psi}_t^{\pi}(\xb^t)$.
\end{proof}
Hence, in what follows, we can use any mixture of achievable wealth processes (or its lower bound) to construct a confidence sequence.

\section{Two-Horse Race and \texorpdfstring{$\{0,1\}$}{\{0,1\}}-Valued Processes}
\label{sec:hr}
In this section, we first consider a $\{0,1\}$-valued sequence $(Y_t)_{t=1}^{\infty}$ such that $\E[Y_t|Y^{t-1}]\equiv\mu$ for some $\mu\in(0,1)$.
This special case well motivates the idea of universal gambling and admits a simpler algorithm than the general case.

We consider the gambling problem with odds vectors $\xb_t \gets [\frac{Y_t}{m},\frac{1-Y_t}{1-m}]$, which is a two-horse race game with odds $o_1=\frac{1}{m}$ and $o_2=\frac{1}{1-m}$.
Henceforth, for the two-horse race with odds $o_1$ and $o_2$, for each $\th\in[0,1]$, we let 
\begin{align*}
\wealth_t^\th(c^t;o_1,o_2)
&\defeq \wealth_0\prod_{i=1}^t (o_1\th)^{c_i}(o_2(1-\th))^{1-c_i}\\
&=\wealth_0(o_1\th)^{\sum_{i=1}^t c_i} (o_2(1-\th))^{t-\sum_{i=1}^t c_i}
\end{align*}
denote the cumulative wealth of the constant bettor $\th$ after round $t$ with respect to $c^t\in\{0,1\}^t$.

\subsection{Motivation}
To illustrate the idea of \emph{universal gambling}, let us first consider the wealth achieved by the best constant bettor \emph{in hindsight}:
\begin{align*}
\max_{\th\in[0,1]} \frac{\wealth_t^\th(c^t;o_1,o_2)}{\wealth_0}
&= o_1^{s_t} o_2^{t-s_t} e^{-th(\frac{s_t}{t})},
\numberthis
\label{eq:best_constant_bettor_wealth}
\end{align*}
where $h(p)\defeq p\log\frac{1}{p}+(1-p)\log\frac{1}{1-p}$ denotes the binary entropy function and $s_t \defeq \sum_{i=1}^t c_i$.
Assume for now that this wealth \scndrevision{were hypothetically achieved} by a causal betting strategy and examine what we can obtain as \scndrevision{a resulting confidence sequence}.
If we plug-in $c_t\gets Y_t$ and $o_1\gets \frac{1}{\mu}$ and $o_2\gets \frac{1}{1-\mu}$, then since the wealth process is a martingale, by Ville's inequality, we have, for any $\d\in(0,1)$,
\begin{align*}
&\P\Bigl\{\sup_{t\ge1}
\mu^{-S_t} (1-\mu)^{-(t-S_t)} e^{-th(\frac{S_t}{t})}
\ge \frac{1}{\d}\Bigr\}\\
&= \P\Bigl\{\sup_{t\ge1}
t d\Bigl(\frac{S_t}{t}~\Big\|~\mu\Bigr)
\ge \log\frac{1}{\d}\Bigr\}
\le \d.
\numberthis\label{eq:best_constant_bettor}
\end{align*}
Here, $S_t\defeq \sum_{i=1}^t Y_i$ and $d(p~\|~q)\defeq p\log\frac{p}{q}+(1-p)\log\frac{1-p}{1-q}$ denotes the binary relative entropy function, \titrevision{and the equality follows from the identity that
\[
d\Bigl(\frac{S_t}{t}~\Big\|~\mu\Bigr)
= -\frac{S_t}{t}\log\mu - \Bigl(1-\frac{S_t}{t}\Bigr)\log(1-\mu) -h\Bigl(\frac{S_t}{t}\Bigr).
\]}%
Since $m\mapsto d(x~\|~m)$ is convex for any $x\in[0,1]$, \eqnref{eq:best_constant_bettor} yields a time-uniform confidence interval.
Further, since $d(p~\|~q)\ge 2(p-q)^2$ by Pinsker's inequality, \eqnref{eq:best_constant_bettor} readily implies
\[
\P\Bigl\{
\exists t\ge 1\suchthat  
\mu \notin 
\Bigl(\frac{S_t}{t} - \sqrt{\frac{1}{2t}\log\frac{1}{\d}},
\frac{S_t}{t} + \sqrt{\frac{1}{2t}\log\frac{1}{\d}}\Bigr)
\Bigr\}
\le \d.
\numberthis\label{eq:hypothetic_confidence_sequence}
\]
Note that this is the time-uniform \scndrevision{version} of the confidence interval implied by Hoeffding's inequality, which states that, for any $\d>0$,
\[
\P\Bigl\{
\mu\notin \Bigl(\frac{S_t}{t}-\sqrt{\frac{1}{2t}\log\frac{2}{\d}},
\frac{S_t}{t}+\sqrt{\frac{1}{2t}\log\frac{2}{\d}}\Bigr)\Bigr\} \le \d
\numberthis\label{eq:hoeffidng}
\]
for any fixed $t\ge 1$.
\titrevision{In other words, achieving the best wealth in hindsight in~\eqnref{eq:best_constant_bettor_wealth} would result in the time-uniform Hoeffding in~\eqnref{eq:hypothetic_confidence_sequence}.}
Therefore, it is reasonable to aim to achieve the best wealth in hindsight in~\eqnref{eq:best_constant_bettor_wealth}, which is the very idea of universal gambling.
\scndrevision{We note in passing that the time-uniform Hoeffding in \eqnref{eq:hypothetic_confidence_sequence} cannot be constructed in reality, as it violates the law of iterated logarithm (LIL)~\citep{Howard--Ramdas--McAuliffe--Sekhon2021}. We refer interested readers to Remark~\ref{rem:asymp_behavior} for further comments on LIL. 
From an online learning perspective, the best wealth in hindsight can only be achieved at the cost of additional regret by a causal betting strategy, and thus the resulting confidence sequence will tend to behave the time-uniform Hoeffding in the long run with an additional slackness due to the regret.}

\subsection{A Universal Gambling Method via the \scndrevision{Method of Mixtures}}
We now show that the \scndrevision{method of mixtures} almost achieves \eqnref{eq:best_constant_bettor_wealth} and essentially implies \eqnref{eq:hypothetic_confidence_sequence}.
We first consider the following wealth process.

\begin{theorem}
\label{thm:wealth_lower_bound_asymmetric_coin_betting}
There exists a causal betting strategy \titrevision{such that for any $o_1,o_2$, the wealth is given as }
$(\wealth_0 \tilde{\phi}_t(\sum_{i=1}^t c_i;o_1,o_2))_{t\ge 0}$,
where we define
\[
\tilde{\phi}_t(x;o_1,o_2)
\defeq o_1^x o_2^{t-x}\frac{B(x+\half,t-x+\half)}{B(\half,\half)}
\numberthis
\label{eq:defn_asymmetric_potential}
\]
for $x\in[0,t]$ and $B(x,x')\defeq\frac{\Gamma(x)\Gamma(x')}{\Gamma(x+x')}$ denotes the beta function.
\end{theorem}
\begin{proof}
Consider a constant bettor $[b,1-b]\in\Delta_2$ parameterized $b\in[0,1]$.
If we define 
$\phi_t^b(x;o_1,o_2)\defeq (o_1 b)^{x}  (o_2(1-b))^{t-x}$ for $x\in[0,t]$,
we can write the wealth for a constant bettor as
\begin{align*}
\wealth_t^b(c^t;o_1,o_2)
&=\wealth_{0} \prod_{i=1}^t (o_1c_i b + o_2(1-c_i)(1-b))\\
&= \wealth_{0} \phi_t^b\Bigl(\sum_{i=1}^tc_i; o_1,o_2\Bigr).
\end{align*}

We now take a mixture of the wealth processes over the constant bettor $b\in[0,1]$. 
We specifically choose the Beta distribution $\pi(b)=B(b|\half,\half)\propto b^{\half}(1-b)^{\half}$ over $b\in[0,1]$ with the parameters $(\half,\half)$. 
By Corollary~\ref{cor:mixture_is_achievable}, the mixture wealth $\int \wealth_t^b(c^t;o_1,o_2)\diff\pi(b)=\wealth_0\tilde{\phi}_t(\sum_{i=1}^t c_i;o_1,o_2)$ is achievable. \titrevision{Note that the equality is a consequence of the identity 
\[
\int \phi_t^b(x;o_1,o_2)\diff\pi(b) = \tilde{\phi}_t(x;o_1,o_2) \quad\text{for any $x\in[0,t]$},
\]
which, in turn, follows from the definition of the Beta function.}
This concludes the proof.
\end{proof}

\begin{remark}
As pointed out earlier, it is not required to explicitly find a sequential betting strategy that achieves the wealth process to construct a confidence sequence.
We remark, however, that the so-called Krichevsky--Trofimov (KT) strategy~\citep{Krichevsky--Trofimov1981} defined as $b_t(c^{t-1})=(\sum_{i=1}^{t-1}c_i+\half)/(t+1)$ achieves \titrevision{the wealth process defined in Theorem~\ref{thm:wealth_lower_bound_asymmetric_coin_betting}}, for \emph{any} $o_1,o_2>0$. 
{It can be readily verified based on the argument of Theorem~\ref{thm:equivalent_condition_achievable_wealth}, but we explicitly explain the special case
in Appendix~\ref{app:kt_strategy} for completeness.}
\end{remark}

By Ville's inequality, the wealth process above can be transformed into a confidence sequence as follows.
\titrevision{A function $f\suchthat\Xc\to\Real$ for a convex set $\Xc\subseteq \Real^D$ is said to be \emph{log-convex} if $x\mapsto\log f(x)$ is convex. Note that any log-convex function is always quasi-convex.}
\begin{theorem}
\label{thm:hr_interval}
Let $(Y_t)_{t=1}^{\infty}$ be a $\{0,1\}$-valued sequence such that $\E[Y_t|Y^{t-1}]\equiv\mu$ for some $\mu\in(0,1)$.
\begin{enumerate}[label=(\alph*)]
\item For any $\d\in(0,1]$, we have
\[
\P\Bigl\{\sup_{t\ge1} \tilde{\phi}_t\Bigl(\sum_{i=1}^t Y_i; \frac{1}{\mu},\frac{1}{1-\mu}\Bigr) 
\ge \frac{1}{\d}\Bigr\}
\le \d.
\]
\item The function $m\mapsto \tilde{\phi}_t(x; \frac{1}{m},\frac{1}{1-m})$ is log-convex and thus the set 
\[
\Cc_t^\pi(y^t)\defeq \Bigl\{m\in[0,1]\suchthat \sup_{1\le i\le t} \tilde{\phi}_i\Bigl(\sum_{j=1}^i y_j; \frac{1}{m},\frac{1}{1-m}\Bigr)
< \frac{1}{\d}\Bigr\}
\]
is an interval $(\mu_t^{\low}(y^t;\d),\mu_t^{\up}(y^t;\d))\subseteq [0,1]$.
Therefore, equivalently, for any $\d\in(0,1)$, we have
\[
\P\{\exists t\ge 1\suchthat  \mu \notin (\mu_t^{\low}(Y^t;\d), \mu_t^{\up}(Y^t;\d))\}
\le \d.
\numberthis\label{eq:hr_interval}
\]
\end{enumerate}
\end{theorem}

\begin{proof}
Let $S_t\defeq \sum_{i=1}^t Y_i$.
Set $\xb_t \gets [\frac{Y_t}{\mu},\frac{1-Y_t}{1-\mu}]$ in the two-horse race, so that $(\xb_t)_{t=1}^{\infty}$ is fair.
Hence, by Proposition~\ref{prop:wealth_is_martingale}, any wealth process out of this game is a martingale.
Since there exists a strategy that guarantees $\wealth_t(Y^t;\frac{1}{\mu},\frac{1}{1-\mu})= \wealth_0 \tilde{\phi}_t(S_t;\frac{1}{\mu},\frac{1}{1-\mu})$ by Theorem~\ref{thm:wealth_lower_bound_asymmetric_coin_betting}, we have
\begin{align*}
&\P\Bigl\{\sup_{t\ge1} \tilde{\phi}_t\Bigl(S_t; \frac{1}{\mu},\frac{1}{1-\mu}\Bigr) \ge \frac{1}{\d}\Bigr\}\\
&= \P\Bigl\{\sup_{t\ge1} \frac{\wealth_t(Y^t;\frac{1}{\mu},\frac{1}{1-\mu})}{\wealth_0} \ge \frac{1}{\d}\Bigr\}
\le \d,
\end{align*}
where the inequality follows from Ville's inequality (Theorem~\ref{thm:ville}).
Further, since $m\mapsto \log \tilde{\phi}_t(x; \frac{1}{m},\frac{1}{1-m})$ is convex by Lemma~\ref{lem:log_convexity}, the equation $\tilde{\phi}_t(x; \frac{1}{m},\frac{1}{1-m})=\frac{1}{\d}$ in terms of $m$ always has two distinct real roots for any $\d\in(0,1)$, and thus the final inequality readily follows.
\end{proof}
\titrevision{We remark that the confidence sequence in this theorem is a special case of the more general confidence sequence derived by universal portfolio in the next section, for $\{0,1\}$-valued sequences; see Remark~\ref{rem:discrete}.}
In practice, $\mu_t^{\low}(x;\d),\mu_t^{\up}(x;\d)$ are the roots of the equation $\tilde{\phi}_t(S_t; \frac{1}{m},\frac{1}{1-m}) 
= \frac{1}{\d}$ over $m\in[0,1]$, and thus can be numerically computed by a 1D root finding algorithm such as the Newton--Raphson iteration; see, \eg \citep[Sec.~8.1]{Solomon2015}. We note that, when the number of observations $t$ is small, it could be that there may exist no root or only one root in $(0,1)$, as shown in the synthetic examples in Fig.~\ref{fig:evolution}.

By Pinsker's inequality, we can find a simple outer bound on the resulting confidence sequence in~\eqnref{eq:hr_interval}.
\titrevision{The proof of the following corollary can be found in Appendix~\ref{app:proofs}.}
\begin{corollary}
\label{cor:hr_interval}
Let $(Y_t)_{t=1}^{\infty}$ be a $\{0,1\}$-valued sequence such that $\E[Y_t|Y^{t-1}]\equiv\mu$ for some $\mu\in(0,1)$.
Then, for any $\d\in(0,1]$, we have
\[
\P\Bigl\{
\exists t\ge 1\suchthat  
\mu \notin 
\Bigl(\frac{S_t}{t} - \sqrt{\frac{g_t(S_t;\d)}{2}},
\frac{S_t}{t} + \sqrt{\frac{g_t(S_t;\d)}{2}} \Bigr)
\Bigr\}
\le \d,
\]
where we define $S_t\defeq \sum_{i=1}^t Y_i$,
\begin{align*}
\qtkt_t(x)&\defeq \frac{B(x+\half,t-x+\half)}{B(\half,\half)},\quad\text{and}
\numberthis\label{eq:qtkt}\\
g_t(x;\d)&\defeq \frac{1}{t}
\log\frac{1}{\d} + \frac{1}{t}\log \frac{e^{-th(\frac{x}{t})}}{\qtkt_t(x)}.
\end{align*}
\end{corollary}

\begin{remark}[An asymptotic behavior]
\label{rem:asymp_behavior}
\titrevision{Using this simplified outer bound, we can argue that the confidence sequence from universal gambling closely emulates the hypothetical time-uniform Hoeffiding in~\eqnref{eq:hypothetic_confidence_sequence}.
For $t$ sufficiently large,} the size of the interval in Corollary~\ref{cor:hr_interval} behaves as
\[
\sqrt{2g_t(S_t;\d)} = \sqrt{\frac{2}{t}\log\frac{1}{\d}
+ \frac{1}{t}\log t + o(1)},
\]
since $\frac{1}{t}\log \frac{1}{\qtkt_t(x)}= h(\frac{x}{t})+\frac{1}{2t}\log  t + o(1)$ by Stirling's approximation, provided that \titrevision{$S_t=\Th(t)$ and $t-S_t=\Th(t)$}.\footnote{By Stirling's approximation, $\log B(x,y)\approx (x-\half)\log x + (y-\half)\log y - (x+y-\half)\log(x+y) + \half\log(2\pi)$ for $x$ and $y$ sufficiently large.}
Compared to \eqnref{eq:hypothetic_confidence_sequence}, it suffers an additional term $O(\sqrt{(\log t)/t})$ in the width of the confidence sequence.
\titrevision{In words, this shows that the universal gambling can indeed implement a time-uniform Hoeffding as expected, with the cost of additional width which vanishes in time.}
We remark in passing that the law of iterated logarithm (LIL)~\citep{Howard--Ramdas--McAuliffe--Sekhon2021} implies that the optimal additional term is in the order of $O(\sqrt{(\log\log t)/t})$, as also commented in \citep{Waudby-Smith--Ramdas2020b}. 
\citet[Appendix~C.2]{Waudby-Smith--Ramdas2020b} proposed a betting strategy based on a ``predictable mixture'' with the so-called stitching technique~\citep{Howard--Ramdas--McAuliffe--Sekhon2021} and analyzed that it achieves the optimal order $O(\sqrt{(\log\log t)/t})$.
\titrevision{Applying a similar idea of \citep{Jun--Orabona2019}, \citet{Orabona--Jun2021} also constructed another mixture portfolio based on a prior inspired by \citep{Robbins1970}, and showed that it achieves the LIL based on a regret analysis for $[0,1]$-valued processes. 
}
\end{remark}

\section{Continuous Two-Horse Race \texorpdfstring{\\}{}and \texorpdfstring{$[0,1]$}{[0,1]}-Valued Processes}
\label{sec:up}
We now consider a more general case where the sequence $Y_t$ is continuous-valued, \ie $Y_t\in[0,1]$.
Likewise in the previous section,
we plug-in $\ct_t\gets Y_t$ with $o_1\gets \frac{1}{\mu}$ and $o_2\gets \frac{1}{1-\mu}$ to construct a subfair odds-vector sequence $(\xv_t)_{t=1}^{\infty}$ for a continuous two-horse race.
Following the same notation, we define
\begin{align}
\wealth_t^b(\ct^t;o_1,o_2)
\defeq \wealth_0\prod_{i=1}^t 
(o_1\ct_i b + o_2(1-\ct_i)(1-b))
\label{eq:cum_wealth}
\end{align}
as the cumulative wealth of the constant bettor $b$ after round $t$ with respect to $\ct^t\in[0,1]^t$.

\titrevision{
Note that we have an immediate lower bound
\[
o_1\ct_i b + o_2(1-\ct_i)(1-b) 
\ge (o_1 b)^{\ct_i} (o_2(1-b))^{1-\ct_i}
\numberthis
\label{eq:jensen}
\]
by Jensen's inequality.
Since the lower bound is in the form of the multiplicative wealth $(o_1b)^{\ct_i}(o_2(1-b))^{1-\ct_i}$ of the discrete two-horse race, the same time-uniform confidence intervals constructed in Theorem~\ref{thm:hr_interval} in the previous section is valid as a loose confidence sequence if we simply plug-in $\ct^t$ in place of $Y^t$. Note that though $Y^t\in\{0,1\}^t$ was assumed in Theorem~\ref{thm:hr_interval}, the function defined in the theorem remains well-defined for continuous-valued sequence $Y^t\in[0,1]^t$.}
This lower bound in~\eqnref{eq:jensen} can be viewed as a reduction to the standard two-horse race game, since the lower bound becomes tight if and only if $\ct_i\in\{0,1\}$.
\titrevision{Albeit this results in a very easy-to-implement confidence sequence for continuous-valued sequences,} it turns out that the Jensen gap in \eqnref{eq:jensen} leads to a loose bound.

A more direct approach for this general case is to compute the mixture wealth without invoking Jensen's inequality.
That is, we will study how to compute the mixture wealth, with a slight abuse of notation,
\begin{align*}
\wealth_t^{\pi}(\ct^t;o_1,o_2)
&\defeq \int \wealth_t^b(\ct^t;o_1,o_2)\diff\pi(b)\\
&\ge \wealth_{0} \tilde{\phi}_t\Bigl(\sum_{i=1}^t \ct_i; o_1,o_2\Bigr),
\end{align*}
where the lower bound is the wealth attained by a mixture portfolio in the previous section and the inequality is followed from \eqnref{eq:jensen}.
Here, the equality holds if and only if $\ct^t\in\{0,1\}^t$.
Hereafter, we call $\wealth_t^{\pi}(\ct^t;o_1,o_2)$ the \emph{$\pi$-mixture wealth}, or the wealth of the \emph{$\pi$-mixture portfolio}.
We remark that
\citet{Cover1991}'s universal portfolio (UP) is a $\mathrm{Beta}(b;\a,\b)$-mixture portfolio; note that $\a=\b=1$ or $\a=\b=\half$ are the canonical choices due to the simplicity~\citep{Cover1991} and the minimax optimality~\citep{Cover--Ordentlich1996}, respectively.
The beta distribution is a natural choice for an ease of implementation, since it is a \emph{conjugate prior} of the summand $b^k(1-b)^{t-k}$, which can be understood as the (unnormalized) binomial distribution.

\subsection{A Mixture Portfolio}
We first explore how we can compute a mixture wealth exactly, \ie without any approximation. 
While \citet{Orabona--Jun2021} alluded to this idea of directly implementing Cover's UP to compute a confidence sequence, here we explicitly describe the algorithm in a general form with any $\pi$-mixture for a reference.
We also show that the resulting confidence set of any mixture portfolio is always an interval as desired. (This was left as a loose end in the initial version of \citep{Orabona--Jun2021} (\href{https://arxiv.org/abs/2110.14099v1}{arXiv:2110.14099v1}), but also resolved in its revision (\href{https://arxiv.org/abs/2110.14099v3}{arXiv:2110.14099v3}).)

It turns out that the $\pi$-mixture wealth can be computed in $O(t)$ at round $t$ instead of $O(2^t)$ via a dynamic-programming type recursion, 
as shown in the following theorem. 
This essentially appeared in the discussion on computing the UP by \citet[Section~IV]{Cover--Ordentlich1996}.
Here we present a general statement for any mixture portfolio.
\begin{theorem}
\label{thm:exact_up}
For any prior distribution $\pi(b)$ over $(0,1)$,
the $\pi(b)$-mixture wealth is achievable and can be written as
\begin{align*}
\wealth_t^{\pi}(\ct^t;o_1,o_2)
=\wealth_0 \sum_{k=0}^t 
o_1^ko_2^{t-k}
\psi_t^{\pi}(k)
\ct^t(k),
\numberthis\label{eq:up_wealth_expression}
\end{align*}
where we 
define
\begin{align*}
\psi_t^{\pi}(k) &\defeq \int_0^1 b^k(1-b)^{t-k}\diff \pi(b),\numberthis\label{eq:mixture_wealth_individual}\\
\ct^t(k)&\defeq \sum_{z^t\in\{0,1\}^t\suchthat k(z^t)=k} \prod_{i=1}^t \ct_i^{z_i} (1-\ct_i)^{1-z_i},\numberthis\label{eq:seq_k_statistics}
\end{align*}
and $k(z^t)\defeq \sum_{i=1}^t z_i$.
Furthermore, for each $t\ge 1$, we have
\[\ct^t(k)=
\begin{cases}
(1-\ct_t)\ct^{t-1}(0) & \text{if }k=0,\\
\ct_t \ct^{t-1}(k-1)+(1-\ct_t)\ct^{t-1}(k) & \text{if }1\le k\le t-1\\
\ct_t \ct^{t-1}(t-1) & \text{if }k=t.
\end{cases}
\numberthis
\label{eq:up_recursive_update}
\]
\end{theorem}

\begin{proof}
We first note that we can write the cumulative wealth of any constant bettor $b$ as
\begin{align*}
&\frac{\wealth_t^b(\ct^t;o_1,o_2)}{\wealth_0}\\
&=\sum_{z^t\in\{0,1\}^t} \prod_{i=1}^t (o_1\ct_i b)^{z_i} (o_2(1-\ct_i)(1-b))^{1-z_i},
\numberthis\label{eq:distributive_law}
\end{align*}
where the equality follows \titrevision{by applying the distributive law to \eqnref{eq:cum_wealth}}.
To see \eqnref{eq:up_wealth_expression}, we first note that continuing from \eqnref{eq:distributive_law}, we have
\begin{align*}
\frac{\wealth_t^b(\ct^t;o_1,o_2)}{\wealth_0}
&= \sum_{k=0}^t 
(o_1b)^{k}(o_2(1-b))^{t-k}\times\\
&\quad\sum_{z^t\in\{0,1\}^t\suchthat k(z^t)=k} \prod_{i=1}^t  \ct_i^{z_i}(1-\ct_i)^{1-z_i},
\numberthis\label{eq:wealth_const_betting}
\end{align*}
and thus integrating over $b$ with respect to $\pi(b)$ leads to \eqnref{eq:up_wealth_expression} by the definition of \eqnref{eq:defn_asymmetric_potential}.
The recursive update equation in~\eqnref{eq:up_recursive_update} is easy to verify.
\end{proof}

\subsubsection{\titrevision{Confidence sequence from mixture portfolio}}
The wealth process of any mixture portfolio yields a confidence sequence as follows, which turns out to be always in the form of confidence \emph{intervals} as desired.
\titrevision{In what follows, we define and denote the empirical mean of $Y^t$ as $\hat{\mu}_t\defeq \frac{1}{t}\sum_{i=1}^t Y_i$.}
\begin{theorem}
\label{thm:up_interval}
Let $(Y_t)_{t=1}^{\infty}$ be a $[0,1]$-valued sequence such that $\E[Y_t|Y^{t-1}]\equiv\mu$ for some $\mu\in(0,1)$.
\begin{enumerate}[label={(\alph*)}]
\item For any $\d\in(0,1]$, we have
\[
\P\Bigl\{\sup_{t\ge1} h_t^{\pi}(\mu;Y^t) \ge \frac{1}{\d}\Bigr\}
\le \d,
\]
where we define
\[h_t^{\pi}(m;y^t)
\defeq \frac{\wealth_t^\pi(y^t;\frac{1}{m},\frac{1}{1-m})}{\wealth_0}
= \sum_{k=0}^t 
\frac{\psi_t^\pi(k) y^t(k)}{m^k(1-m)^{t-k}}
\]
and $y^t(k)$ is defined analogously to \eqnref{eq:seq_k_statistics}. 
\item The function $m\mapsto h_t^{\pi}(m;y^t)$ is log-convex and thus the set 
\[\tilde{\Cc}_t^{\pi}(y^t;\d)\defeq \Bigl\{m\in[0,1]\suchthat \sup_{1\le i\le t} h_i^{\pi}(m;y^i)< \frac{1}{\d}\Bigr\}\]
is an interval $(\tilde{\mu}_t^{\pi,\low}(y^t;\d),\tilde{\mu}_t^{\pi,\up}(y^t;\d))\subseteq[0,1]$.
Therefore, equivalently, for any $\d\in(0,1)$, we have
\[
\P\{\exists t\ge 1\suchthat  \mu \notin (\tilde{\mu}_t^{\pi,\low}(Y^t;\d), \tilde{\mu}_t^{\pi,\up}(Y^t;\d))\}
\le \d.
\]
\item $\hat{\mu}_t\in \tilde{\Cc}_t^{\pi}(Y^t)$ with probability 1.
\end{enumerate}
\end{theorem}
\begin{proof}
The proof of part (a) follows the same line of that of Theorem~\ref{thm:hr_interval}, with applying Theorem~\ref{thm:exact_up}. 
For part (b), note that $m\mapsto m^{-k}(1-m)^{-(t-k)}$ is log-convex for each $k=0,\ldots,t$ by Lemma~\ref{lem:log_convexity}, a sum of any log-convex functions is also log-convex (Lemma~\ref{lem:sum_preserves_log_convexity}), and so is the function $m\mapsto h_t^{\pi}(m;y^t)$ as a function of $m$. 
Finally, the fact that $\hat{\mu}_t\in \tilde{\Cc}_t^{\pi}(Y^t)$ almost surely readily follows from that $h_t^\pi(\hat{\mu}_t;y^t)\le 1\le \frac{1}{\d}$ by Lemma~\ref{lem:small_wealth}, stated below. 
\end{proof}

\titrevision{The last property (c) that the resulting confidence set $\Cc_t$ always contains the empirical mean $\hat{\mu}_t$ will be proven useful for a lower-bound approach in the next section.
An interesting lemma used in proving (c) is that no constant bettor can earn any money from the continuous two-horse race with odds $(\frac{1}{\hat{\mu}_t},\frac{1}{1-\hat{\mu}_t})$ with any underlying sequence $\ct^t$, at any round, \emph{deterministically}.
Its proof is deferred to Appendix~\ref{app:proofs}.
\begin{lemma}
\label{lem:small_wealth}
For $\ct^t\in[0,1]^t$, let $\hat{\mu}_t=\frac{1}{t}\sum_{i=1}^t \ct_i$.
For any $b\in[0,1]$, we have
\[
\wealth_t^b\Bigl(\ct^t;\frac{1}{\hat{\mu}_t},\frac{1}{1-\hat{\mu}_t}\Bigr)
\le \wealth_0.
\]
\end{lemma}}

\subsubsection{\titrevision{Implementation and complexity}}
Similar to Theorem~\ref{thm:hr_interval}, we can apply the Newton--Raphson iteration or the bisect method to numerically compute the confidence intervals. 
At each round $t$, however, it takes $O(Mt)$ time complexity if $M$ is the maximum number of iterations in a numerical root-finding method, since evaluating $h_t^{\pi}(\mu;y^t)$ for a given $\mu$ takes $\Th(t)$.
As a result, to process a sequence of length $T$, it takes a quadratic complexity $O(T^2)$.

As alluded to earlier, \citet{Orabona--Jun2021} studied the wealth process of Cover's UP which is the $\mathrm{Beta}(\half,\half)$-mixture portfolio, and considered a tight lower bound based on a regret analysis of UP.
They first noted that, for any $b\in[0,1]$,
\begin{align*}
\frac{\wealthup_t(\ct^t;\frac{1}{m},\frac{1}{1-m})}{\wealth_t^b(\ct^t)}
&\ge \min_{x\in[0,t]} \frac{\qtkt_t(x)}{\qt_t^b(x)},
\end{align*}
where $\qt_t^b(x)\defeq b^x(1-b)^{t-x}$ and $\qtkt_t(x)$ is defined in \eqnref{eq:qtkt}. 
They then considered the following lower bound
\begin{align*}
&\wealthup_t\Bigl(\ct^t;\frac{1}{m},\frac{1}{1-m}\Bigr)\\
&\ge \max_{b\in[0,1]} \wealth_t^b\Bigl(\ct^t;\frac{1}{m},\frac{1}{1-m}\Bigr)
\min_{x\in[0,t]} \frac{\qtkt_t(x)}{\qt_t^b(x)}\\
&\ge  
\wealth_t^{b^*(\ct^t)}\Bigl(\ct^t;\frac{1}{m},\frac{1}{1-m}\Bigr)
\min_{x\in[0,t]} \frac{\qtkt_t(x)}{\qt_t^{b^*(\ct^t)}(x)},
\numberthis\label{eq:precise_lower_bound}
\end{align*}
where $b^*(\ct^t)\defeq\argmax_{b\in[0,1]} \wealth_t^b(\ct^t)$.
The second term of \eqnref{eq:precise_lower_bound} can be bounded by a closed-form expression of the minimax regret (see \citep[Theorem~5]{Orabona--Jun2021}), and thus, in principle, solving the maximization problem to find $b^*(\ct^t)$ leads to a confidence sequence.
Since, however, the lower bound is not a quasiconvex function as a function of $m$, \citet{Orabona--Jun2021} came up with an algorithm, dubbed as PRECiSE-CO96, which numerically computes the intervals based on the final lower bound in~\eqnref{eq:precise_lower_bound}. 
While this algorithm gives empirically tighter intervals for small-sample regime and comparable performance in a long run compared to the existing methods of \citet{Howard--Ramdas--McAuliffe--Sekhon2021}, the algorithm and analysis are rather complicated by nature. 
\titrevision{Moreover, the PRECiSE-CO96 still suffers the same per-round time complexity $O(t)$ of Cover's UP, since when computing $b^*(\ct^t)$, maximizing $\wealth_t^b(\ct^t)$ necessitates to evaluate the function which has $t$ summands. \titrevision{They also proposed another algorithm PRECiSE-A-CO96 as a relaxation of PRECiSE-CO96, by invoking \scndrevision{a version of an inequality due to \citet{Fan--Grama--Liu2015}} on their final lower bound, which runs in $O(1)$ per round.}
We propose a different approach of constant per-round complexity in the next section.}

\begin{remark}[A special case of discrete processes]
\label{rem:discrete}
\titrevision{For a discrete sequence $\ct^t\in \{0,1\}^t$, it is easy to check that the cumulative wealth in~\eqnref{eq:cum_wealth} simplifies to
\[
\sum_{k=0}^t 
o_1^ko_2^{t-k}
\psi_t^{\pi}(k)
\ct^t(k)
= o_1^{\sum_{i=1}^t \ct_i}o_2^{t-\sum_{i=1}^t \ct_i}
\psi_t^{\pi}\Bigl(\sum_{i=1}^t \ct_i\Bigr).
\]
This implies that the mixture wealth in~\eqnref{eq:cum_wealth} with respect to a discrete sequence is equivalent to that of the standard two-horse race as expected, and thus we can use the simpler two-horse-race-based method for $\{0,1\}$-valued processes without invoking the more general algorithm in Theorem~\ref{thm:exact_up}. 
Since the horse-race-based method takes $O(T)$ complexity for a sequence of length $T$, while the general algorithm may suffer $O(T^2)$ as elaborated below, it is worth treating discrete processes separately as a special case.
We note that this observation was not realized in \citep{Orabona--Jun2021}.}
\end{remark}%

\subsection{\titrevision{A Mixture-of-Lower-Bounds Portfolio}}
In this section, we propose a straightforward and computationally efficient alternative to a mixture portfolio, based on a new tight lower bound to the wealth of a constant bettor.
\titrevision{To motivate our development, two remarks on the computational aspects on Cover's UP are in order. We first note that the choice of the Beta distribution $\mathrm{Beta}(b;\a,\b)$ as a mixture distribution $\pi(b)$ is closely related to the fact that the Beta distribution is the conjugate prior of the (unnormalized) binomial distribution $b^k(1-b)^{t-k}$ (see \eqnref{eq:mixture_wealth_individual}), where the specific choice $\a=\b=\half$ is for its minimax optimality.
Second, we remark that the computational complexity of UP is mainly in computing the term in~\eqnref{eq:seq_k_statistics}, which originates from the alternative wealth expression in~\eqnref{eq:distributive_law} followed by the distributive law applied on \eqnref{eq:cum_wealth}.
Recall that \citet{Orabona--Jun2021} considered a lower bound on the wealth of UP to circumvent the computational complexity.}

\titrevision{Alternatively, we propose to apply a tight lower bound on the cumulative wealth in~\eqnref{eq:cum_wealth} of constant bettors first, and then consider a mixture of the lower bounds. In this way, we can detour applying the distributive law and thus avoid the costly recursive expression in~\eqnref{eq:seq_k_statistics}, resulting in a wealth lower bound that can be computed in constant complexity.
Note that it is still a valid approach for constructing a confidence sequence as PRECiSE is, since 
we can always use a \emph{deterministic} lower bound on the wealth process in \eqnref{eq:apply_ville}, as it only results in a larger (and thus worse) confidence set.}

\titrevision{In the following, we first state a new lower bound to be applied on \eqnref{eq:cum_wealth}. Since a mixture of the lower bounds is considered, we define a new mixture distribution in the form of a conjugate prior, by regarding the form of the lower bounds as an unnormalized exponential family distribution. While the use of conjugate prior is mathematically analogous to the choice of the Beta distribution in Cover's UP, we remark that our choice of prior is purely based on computational considerations, and we do not claim any optimality of this choice, in contrast to the optimal choice of $\mathrm{Beta}(b;\half,\half)$ in Cover's UP.}

\subsubsection{\titrevision{A new lower bound on a logarithmic function}}
We start with the lemma below, which provides a way to approximate a logarithmic function from below with polynomial functions.
This is a generalization of \citep[Lemma~1]{Waudby-Smith--Ramdas2020b}, which is subsumed by the case with $n=1$. \titrevision{The proof can be found in Appendix~\ref{app:proofs}.}
In what follows, for the sake of simple notation, for some $x\in[0,1]$, we let $\bar{x}\defeq 1-x$.

\begin{table*}
\caption{Definition of $f_n(y,b,m)$ in Lemma~\ref{lem:generalized_lower_bound}.}
\centering
\begin{minipage}{\textwidth}
\begin{align}
f_n(y,b,m)\defeq \begin{cases}
\displaystyle
\sum_{k=1}^{2n-1} \frac{1}{k} \Bigl(1-\frac{\bar{b}}{\bar{m}}\Bigr)^k \Bigl\{\Bigl(1-\frac{y}{m}\Bigr)^{2n}
    -\Bigl(1-\frac{y}{m}\Bigr)^k\Bigr\} + \Bigl(1-\frac{y}{m}\Bigr)^{2n}\log \frac{\bar{b}}{\bar{m}} 
    & \text{if $b\in[m,1)$, $y\ge 0$,}\\
\displaystyle
\sum_{k=1}^{2n-1} \frac{1}{k} \Bigl(1-\frac{b}{m}\Bigr)^k \Bigl\{\Bigl(1-\frac{\bar{y}}{\bar{m}}\Bigr)^{2n}
    -\Bigl(1-\frac{\bar{y}}{\bar{m}}\Bigr)^k\Bigr\} + \Bigl(1-\frac{\bar{y}}{\bar{m}}\Bigr)^{2n}\log \frac{b}{m} 
    & \text{if $b\in(0,m]$, $y\le 1$.}
\end{cases}
\label{eq:long_eq1}
\end{align}
\medskip
\hrule
\end{minipage}
\label{tab:long_eq1}
\end{table*}

\begin{table*}
\caption{An alternative expression of Lemma~\ref{lem:generalized_lower_bound} with the definition of the exponential-family distribution in \eqnref{eq:exp_family}.}
\centering
\begin{minipage}{\textwidth}
\begin{align}
b\frac{y}{m} + \bar{b}\frac{\bar{y}}{\bar{m}} 
&\ge \begin{cases}
\displaystyle
\phi_n\Bigl(\frac{\bar{b}}{\bar{m}};\Bigl(\Bigl(1-\frac{y}{m}\Bigr)^{2n} - \Bigl(1-\frac{y}{m}\Bigr)^{k}\Bigr)_{k=1}^{2n-1}, \Bigl(1-\frac{y}{m}\Bigr)^{2n}\Bigr)
    & \text{if $b\in[m,1)$, $y\ge 0$,}\\
\displaystyle
\phi_n\Bigl(\frac{b}{m};\Bigl(\Bigl(1-\frac{\bar{y}}{\bar{m}}\Bigr)^{2n} - \Bigl(1-\frac{\bar{y}}{\bar{m}}\Bigr)^{k}\Bigr)_{k=1}^{2n-1}, \Bigl(1-\frac{\bar{y}}{\bar{m}}\Bigr)^{2n}\Bigr)
    & \text{if $b\in(0,m]$, $y\le 1$.}
\end{cases}
\label{eq:long_eq2}
\end{align}
\medskip
\hrule
\end{minipage}
\label{tab:long_eq2}
\end{table*}

\begin{table*}
\caption{Definition of $\breve{h}_t^{(n)}(m;y^t;\rhob_0^{(1)},\eta_{0}^{(1)},\rhob_0^{(2)},\eta_{0}^{(2)})$ in Theorem~\ref{thm:lbup}.}
\centering
\begin{minipage}{\textwidth}
\begin{align*}
&\breve{h}_t^{(n)}(m;y^t;\rhob_0^{(1)},\eta_{0}^{(1)},\rhob_0^{(2)},\eta_{0}^{(2)})
\defeq \frac{\bar{m} Z_n(\rhob_n(y^t;m)+\rhob_0^{(1)},\eta_{n}(y^t;m)+\eta_0^{(1)}) 
+ m Z_n(\rhob_n(\bar{y}^t;\bar{m})+\rhob_0^{(2)},\eta_{n}(\bar{y}^t;\bar{m})+\eta_0^{(2)})}
{\bar{m} Z_n(\rhob_0^{(1)},\eta_{0}^{(1)}) 
+ m Z_n(\rhob_0^{(2)},\eta_{0}^{(2)})}.
\numberthis\label{eq:lbup_wealth}
\end{align*}
\medskip
\hrule
\end{minipage}
\label{tab:long_eq3}
\end{table*}

\begin{lemma}
\label{lem:generalized_lower_bound}
For any $n\in \Natural$ and $m\in(0,1)$, we have
\begin{align*}
\log\Bigl(b\frac{y}{m} + \bar{b}\frac{\bar{y}}{\bar{m}} \Bigr)
&\ge f_n(y,b,m),
\end{align*}
where $f_n(y,b,m)$ is defined in \eqnref{eq:long_eq1}.
\end{lemma}
\titrevision{We remark that the positive integer $n\ge1$ can be understood as an order of approximation, where a higher order $n$ would result in a tighter lower bound. 
We conjecture that this bound becomes monotonically tighter as $n$ increases, but we do not have a formal proof.
Our experiments empirically show, however, that using the lower bound with larger $n$ leads to monotonically tighter confidence sequences; see Section~\ref{sec:exp:evolution} and Fig.~\ref{fig:evolution} therein.
A trade-off in choosing $n$ is discussed in Remark~\ref{rem:tradeoff} below.}

As alluded to earlier, the form of the \titrevision{(exponentiated)} lower bound leads us to defining an exponential-family distribution $p_n(x;\rhob,\eta)$ over $x\in(0,1)$ with parameters $\rhob\in \Real^{2n-1}$ and $\eta\ge 0$ as
\begin{align*}
p_n(x;\rhob,\eta)&\defeq \frac{\phi_n(x;\rhob,\eta)}{Z_n(\rhob,\eta)},\numberthis\label{eq:exp_family}\\
\log \phi_n(x;\rhob,\eta)
&\defeq\sum_{k=1}^{2n-1} \frac{(1-x)^k}{k} \rho_k + \eta\log x,
\end{align*}
where \[Z_n(\rhob,\eta)\defeq \int_0^1 \phi_n(x;\rhob,\eta)\diff x\] 
is the normalization constant. 
Note that if $\rho_1=\ldots=\rho_{2n-1}=\eta=0$, it becomes a uniform distribution over $[0,1]$.
Further, for $n=1$, if $\rho_1\ge 0$, we can write \[Z_1(\rho_1,\eta)= e^{\rho_1}\rho_1^{-\eta-1}\gamma(\eta+1,\rho_1),
\numberthis\label{eq:normal_cont_special}
\] where $\gamma(s,x)\defeq \int_0^x t^{s-1}e^{-t}\diff t$ for $s>0$ denotes the lower incomplete gamma function.
Moreover, we use the bar notation $\bar{x}\defeq 1-x$ for any $x\in[0,1]$ in what follows.

With these definitions, the lower bound in Lemma~\ref{lem:generalized_lower_bound} can be succinctly rewritten as in \eqnref{eq:long_eq2}.

Since it is easy to check 
\[\phi_n(x;\rhob,\eta) \phi_n(x;\rhob_0,\eta_0)
=\phi_n(x;\rhob+\rhob_0,\eta+\eta_0)\numberthis\label{eq:additive}\] 
by definition, we obtain the following lower bound to the cumulative wealth of a constant bettor.

\begin{lemma}
\label{lem:generalized_lower_bound_product}
For any $n\in\Natural$, $m\in(0,1)$, $b\in[0,1]$, and $y^t\in[0,1]^t$, we have
\begin{align*}
\frac{\wealth_t^b(y^t;\frac{1}{m},\frac{1}{\bar{m}})}{\wealth_0}
&\ge \phi_n\Bigl(\frac{\bar{b}}{\bar{m}};\rhob_n(y^t;m),\eta_{n}(y^t;m)\Bigr)\ones_{(m,1)}(b)\\
&~+ \phi_n\Bigl(\frac{b}{m};
\rhob_n(\bar{y}^t;\bar{m}),\eta_{n}(\bar{y}^t;\bar{m})\Bigr)\ones_{(0,m]}(b),
\end{align*}
where we define 
$\bar{y}^t\defeq (1-y_i)_{i=1}^t$,
\begin{align*}
(\rhob_{n}(y^t;m))_k
&\defeq \sum_{i=1}^t \Bigl\{
\Bigl(1-\frac{y_i}{m}\Bigr)^{2n} - \Bigl(1-\frac{y_i}{m}\Bigr)^{k}\Bigr\}
\end{align*}
for $k=1,\ldots,2n-1$, and
\begin{align*}
\eta_n(y^t;m)
&\defeq \sum_{i=1}^t \Bigl(1-\frac{y_i}{m}\Bigr)^{2n}.
\end{align*}
\end{lemma}

Note that $\rhob_{n}(y^t;m)$ and $\eta_{n}(y^t;m)$ can be updated \emph{sequentially} without storing the entire history $y^t$. 
To see this, define $s_j(y^t)
\defeq \sum_{i=1}^t y_i^j$ for $y^t\in[0,1]^t$.
We can then rewrite the above expressions as
\begin{align*}
(\rhob_{n}(y^t;m))_k
&= \sum_{j=0}^{2n} \binom{2n}{j} \frac{s_{j}(y^t)}{(-m)^{j}}
- \sum_{j=0}^k \binom{k}{j} \frac{s_j(y^t)}{(-m)^j}
\end{align*}
for $k=1,\ldots,2n-1$, and
\begin{align*}
\eta_n(y^t;m)
&=\sum_{j=0}^{2n} \binom{2n}{j} \frac{s_j(y^t)}{(-m)^j}.
\end{align*}

\subsubsection{\titrevision{Lower-bound universal portfolio}}
\titrevision{We now propose to consider a mixture of the lower bounds over $b\in[0,1]$.}
For the sake of computational tractability,
for $\rhob_0^{(1)}\in\Real^{2n-1},\eta_{0}^{(1)}\in\Real,\rhob_0^{(2)}\in\Real^{2n-1},\eta_{0}^{(2)}\in\Real$,
we can define a \emph{conjugate prior} over $b\in[0,1]$ that corresponds to the lower bound defined in Lemma~\ref{lem:generalized_lower_bound_product} as
\begin{align*}
&\pi_m^{(n)}(b;\rhob_0^{(1)},\eta_{0}^{(1)},\rhob_0^{(2)},\eta_{0}^{(2)})
\numberthis\label{eq:general_prior}
\\
&\defeq \frac{\phi_n(\frac{\bar{b}}{\bar{m}};\rhob_0^{(1)},\eta_0^{(1)})\ones_{(m,1)}(b) + \phi_n(\frac{b}{m};\rhob_0^{(2)},\eta_0^{(2)})\ones_{(0,m]}(b)}{\bar{m} Z_n(\rhob_0^{(1)},\eta_{0}^{(1)})
+ m Z_n(\rhob_0^{(2)},\eta_{0}^{(2)})}.
\end{align*}
We will use this prior distribution to compute a mixture,
which is of a different shape compared to the beta prior of Cover's UP in general. 
In particular, however, if we set $\rhob_{0}^{(i)}=0,\eta_{0}^{(i)}=0$ for each $i=1,2$, then the prior $\pi_m^{(n)}(b)$ boils down the uniform distribution over $[0,1]$ for any $m\in[0,1]$, and the resulting mixture wealth lower bound can be viewed as a lower bound to Cover's UP with the uniform prior (\ie $B(1,1)$ prior).
\titrevision{We thus call the resulting wealth lower bound the \emph{lower-bound UP wealth of approximation order $n$}, or LBUP($n$) in short.}

\begin{theorem}
\label{thm:lbup}
Let $n\ge 1$ and $m\in(0,1)$.
If $\pi(b)=\pi_m^{(n)}(b;\rhob_0^{(1)},\eta_{0}^{(1)},\rhob_0^{(2)},\eta_{0}^{(2)})$, for any $y^t\in[0,1]^t$,
we have
\begin{align*}
\frac{\wealth_t^\pi(y^t;\frac{1}{m},\frac{1}{\bar{m}})}{\wealth_0}
\ge 
\breve{h}_t^{(n)}(m;y^t;\rhob_0^{(1)},\eta_{0}^{(1)},\rhob_0^{(2)},\eta_{0}^{(2)}),
\end{align*}
where we define $\breve{h}_t^{(n)}(m;y^t;\rhob_0^{(1)},\eta_{0}^{(1)},\rhob_0^{(2)},\eta_{0}^{(2)})$ in \eqnref{eq:lbup_wealth}.
\end{theorem}
\begin{proof}
By the aforementioned property in~\eqnref{eq:additive},
the lower bound readily follows from the definition of the prior in~\eqnref{eq:general_prior} and the lower bound of Lemma~\ref{lem:generalized_lower_bound_product}.
\end{proof}

The wealth lower bound can be viewed as a discrete mixture, which is called a \emph{hedged} betting in \citep{Waudby-Smith--Ramdas2020b}, over two continuous mixture wealth processes.
Here, the first term $Z_n(\rhob_n(y^t;m),\eta_{n}(y^t;m))$ corresponds to the mixture of constant bettors $b\in[m,1)$ and the second term corresponds to $b\in(0,m]$, and they are weighted by $(1-m)$ and $m$, respectively.

\subsubsection{\titrevision{Confidence sequence from LBUP}}
We now state the resulting confidence sequence from this lower bound approach, which is our main technical result.
It is worth noting that since any mixture wealth process has an interval guarantee (Theorem~\ref{thm:up_interval}), the lower-bound approach can be easily made to inherit it, by choosing the largest interval that contains the empirical mean.

\begin{corollary}
\label{cor:lbup_interval}
Let $(Y_t)_{t=1}^{\infty}$ be a $[0,1]$-valued sequence such that $\E[Y_t|Y^{t-1}]\equiv\mu$ for some $\mu\in(0,1)$.
Let $(\breve{\mu}_t^\low(y^t;\d),\breve{\mu}_t^\up(y^t;\d))$ be the largest interval such that
\begin{align*}
\hat{\mu}_t
&\in(\breve{\mu}_t^\low(y^t;\d),\breve{\mu}_t^\up(y^t;\d))
\numberthis
\label{eq:largest_mid_interval}\\
&\subseteq
\Bigl\{m\in(0,1)\suchthat \breve{h}_t^{(n)}(m;y^t;\rhob_0^{(1)},\eta_{0}^{(1)},\rhob_0^{(2)},\eta_{0}^{(2)})
< \frac{1}{\d}\Bigr\},
\end{align*}
where $\hat{\mu}_t\defeq\frac{1}{t}\sum_{i=1}^t y_i$ denotes the empirical mean.
Then, for any $\d\in(0,1]$, we have
\[
\P\{\exists t\ge 1\suchthat \mu \notin (\breve{\mu}_t^\low(Y^t;\d),\breve{\mu}_t^\up(Y^t;\d))\}
\le \d.
\]
\end{corollary}
\begin{proof}
Since the lower bound lower-bounds the mixture wealth and the mixture wealth always guarantees a sublevel set to be an interval, we can guarantee that the sublevel set defined by the lower bound always subsume the confidence set $\tilde{\Cc}_t^\pi(y^t)$ induced by the mixture wealth.
Further, $\hat{\mu}_t\in \tilde{\Cc}_t^\pi(y^t)$ by Theorem~\ref{thm:up_interval}(c), $(\breve{\mu}_t^\low(y^t;\d),\breve{\mu}_t^\up(y^t;\d))$ satisfying \eqnref{eq:largest_mid_interval} must be a valid uniform confidence interval (at level $1-\d$).
\end{proof}

A reasonable choice of the hyperparameters for the prior is all-zero, \ie $\rhob_{0}^{(i)}=0,\eta_{0}^{(i)}=0$ for each $i=1,2$.
In Section~\ref{sec:exps}, we plot the evolution of the wealth processes of UP and LBUP with simulated datasets, to illustrate the power of the lower-bounds approach; see Fig.~\ref{fig:evolution}.

\subsubsection{\titrevision{Implementation and complexity}}
\titrevision{Our implementation employs the bisect method to compute the confidence sequence with the approach, as the Newton--Raphson iteration is hard to apply as the derivative of the LBUP wealth with respect to $m$ is difficult to compute.}

\titrevision{We remark that to compute the lower bound of Theorem~\ref{thm:lbup} for any value of $m\in(0,1)$, it suffices to keep track of the (unnormalized) empirical moments $(s_j(y^t))_{j=0}^{2n}$, a constant number of real values.
\titrevision{As a result, the per-step time complexity is $O(n)$ and the storage complexity is $O(n)$ for any time step.}
This is in sharp contrast to Cover's UP (as shown in Theorem~\ref{thm:exact_up}) and \titrevision{PRECiSE-CO96} of \citep{Orabona--Jun2021}, which require to store the entire history $Y^t$ and $O(t)$-time complexity at round $t$.}

\begin{remark}[An approximation--computation trade off in the approximation order]
\label{rem:tradeoff}
The lower bound in~\eqnref{eq:lbup_wealth} with the parameter $n\ge 1$ can be understood as an approximation using empirical moments $(s_j(y^t))_{j=1}^{2n}$.
\titrevision{As explained after Lemma~\ref{lem:generalized_lower_bound}, we empirically show that the lower bound with larger $n$ leads to tighter confidence sequences.}
Note, however, that there exists a downside for using a large $n$ in practice.
To implement the wealth lower bound in Theorem~\ref{thm:lbup}, we need to compute the normalization constant $Z_n(\rhob,\eta)$, which is of the form
\[
\int_0^1 x^\eta \exp\Bigl(\sum_{k=0}^{2n-1} a_k x^k\Bigr)\diff x.
\]
In general, there is no closed-form expression for this kind of integrals except the special case of $n=1$ with $\rho_1\ge 0$; see \eqnref{eq:normal_cont_special}.
It is thus necessitated to rely on numerical integration methods\footnote{In the experiments, we used the \texttt{scipy.integration.quad} method of the python package \texttt{scipy}~\citep{2020SciPy-NMeth}.}, but a higher-order $n$ may lead to numerical instability (especially for $m\approx 0$ and $m\approx 1$) and longer time to compute the integral. 
We empirically found that $n=3$ already closely approximates the performance of Cover's UP while giving a decent improvement to $n=1$ or 2, and using a larger order usually provides only marginal gain. 
\end{remark}

\subsubsection{A hybrid approach for the best of both worlds}
On the one hand, while the UP approach becomes time-consuming as a sequence gets longer, it demonstrates a great empirical performance especially for a small-sample regime.
On the other hand, the LBUP approach performs almost identically to the UP approach with only constant per-round complexity, but it requires some burn-in samples to be tight enough as a lower bound as demonstrated below in Fig.~\ref{fig:evolution}.

To achieve the best of the both worlds, we can hybrid the UP and LBUP approaches: namely, we run the UP for an initial few rounds, then we run the suboptimal lower-bound UP afterwards.
\titrevision{We emphasize that, in principle, any gambling strategy can be run in the first part, but we specifically use UP for its outstanding performance.}

\titrevision{Formally, the performance of the hybrid approach can be stated as follows:
\begin{corollary}
\label{cor:hybridup}
Consider the gambling strategy that runs UP for the first $t_{\texttt{UP}}$ steps and switches to run LBUP afterwards.
If the LBUP strategy uses the prior of the form
\[\pi_m^{(n)}(b;\rhob_n(y^{t_{\texttt{UP}}};m),\eta_n(y^{t_{\texttt{UP}}};m),\rhob_n(\bar{y}^{t_{\texttt{UP}}};\bar{m}),\eta_n(\bar{y}^{t_{\texttt{UP}}};\bar{m})),
\numberthis\label{eq:hybrid_lbup_prior}
\] 
then this hybrid gambling strategy attains the cumulative wealth defined in \eqnref{eq:hybrid_wealth} at time step $t\ge t_{\texttt{UP}}$.
\end{corollary}
We remark that the choice of prior in~\eqnref{eq:hybrid_lbup_prior} is natural as it is induced by the first part of the sequence $y^{t_{\texttt{UP}}}$.
We empirically demonstrate the effectiveness of this hybrid approach in Section~\ref{sec:exps}.}

\begin{table*}
\caption{Definition of the cumulative wealth attained by the hybrid strategy in Corollary~\ref{cor:hybridup}.}
\centering
\begin{minipage}{\textwidth}
\begin{align}
\wealthup_{t_{\texttt{UP}}}\Bigl(y^{t_{\texttt{UP}}};\frac{1}{m},\frac{1}{\bar{m}}\Bigr)
\cdot \frac{\bar{m} Z_n(\rhob_n(y^t;m),\eta_{n}(y^t;m)) 
+ m Z_n(\rhob_n(\bar{y}^t;\bar{m}),\eta_{n}(\bar{y}^t;\bar{m}))}
{\bar{m} Z_n(\rhob_n(y^{t_{\texttt{UP}}};m),\eta_n(y^{t_{\texttt{UP}}};m)) 
+ m Z_n(\rhob_n(\bar{y}^{t_{\texttt{UP}}};\bar{m}),\eta_n(\bar{y}^{t_{\texttt{UP}}};\bar{m}))}.
\label{eq:hybrid_wealth}
\end{align}
\medskip
\hrule
\end{minipage}
\label{tab:long_eq4}
\end{table*}

\section{Related Work}
\label{sec:related}

\titrevision{
Historically,
\citet{Darling--Robbins1967} first introduced the notion of confidence sequence, and the idea was further studied by \citet{Hoeffding1963} and \citet{Lai1976}.
{A recent line of work such as \citep{Ramdas--Ruf--Larsson--Koolen2020,Waudby-Smith--Ramdas2020b,Howard--Ramdas--McAuliffe--Sekhon2021,Jun--Orabona2019,Orabona--Jun2021} has brought this idea back to researchers' attention. 
} 
The concept of confidence sequences has a close connection to ``safe testing''~\citep{Grunwald--de-Heide--Koolen2020}.
The idea of gambling for confidence sequences (or equivalently, time-uniform concentration inequalities) can be found in several different areas.
Shafer and Vovk have advocated to use the idea of gambling to establish a game-theoretic theory of probability~\citep{Shafer--Shen--Vereshchagin--Vovk2011,Shafer--Vovk2019}. 
Building upon the gambling approach,
\citet{Shafer2021} proposed to test by betting.
\citet{Waudby-Smith--Ramdas2020b} explored the gambling-based approach for constructing confidence sequences in a great detail, and derived and analyzed various betting algorithms.
\citet{Hendriks2021} also explored a similar idea of using betting algorithms and numerically inverting the wealth processes into confidence intervals. 
\citet{Jun--Orabona2019} proposed to use coin betting for developing parameter-free online convex optimization algorithms using the notion of \emph{regret}, and constructed a time-uniform concentration inequality that achieves the law of iterated logarithm for sub-Gaussian random vectors in Banach spaces; see \citep[Section~7.2]{Jun--Orabona2019}.
\citet{Orabona--Jun2021} examined the idea of \citet{Cover1991}'s universal portfolio and its regret analysis to construct tight confidence sequences based on a numerical method.}
For a more detailed literature survey on the idea of gambling for confidence sequences, we refer an interested reader to \citep[Appendix~D]{Waudby-Smith--Ramdas2020b} and \citep[Section~2]{Orabona--Jun2021}, and references therein.

\citet{Waudby-Smith--Ramdas2020a} studied the problem of estimating the mean of a set of deterministic numbers via random sampling without replacement.
As illustrated by \citet{Waudby-Smith--Ramdas2020b}, any construction of a confidence sequence under the current problem setting (including thus the UP and LBUP algorithms) can be easily converted into a confidence sequence for
the sampling without replacement problem.

The algorithms of \citet{Orabona--Jun2021} (dubbed as PRECiSE-CO96, PRECiSE-A-CO96, and PRECiSE-R70) based on the regret analysis were inspired by \citet{Rakhlin--Sridharan2017}, who showed an essential equivalence between the regret guarantee of certain online learning algorithms and concentration inequalities.
In contrast to such a regret-based approach, the current paper illustrates that a sharp concentration can be obtained directly via a deterministic lower bound of supermartingales, without invoking the notion of regret.
We acknowledge, however, that a regret analysis might be essential to analyze the resulting concentration inequality in a closed form, as was done therein.
{PRECiSE-A-CO96, a constant-complexity version of PRECiSE-CO96, used a similar idea of lower bounding the multiplicative wealth by an inequality of \citet{Fan--Grama--Liu2015}.
The difference in our work is twofold. 
First, we rely on Lemma~\ref{lem:generalized_lower_bound}, which is a higher-order-moments generalization of the lower bound; see Appendix~\ref{app:log_lower_bounds}.
In the next section,
we also empirically show that the LBUP algorithm with $n>1$ clearly improves LBUP with $n=1$, which corresponds to the technique of \citet{Fan--Grama--Liu2015} and thus \citet{Orabona--Jun2021}.
Second, rather than maximizing over the betting $b\in[0,1]$ and invoking the regret bound, we take a mixture over the lower bound as if it were the target wealth process. 
While conceptually being simpler in the sense that it does not invoke a regret bound, the downside of our mixture approach is the computational bottleneck in the numerical integration step; see Remark~\ref{rem:tradeoff}.
We finally remark that a mix-and-match approach of our higher-order-moments bound (Lemma~\ref{lem:generalized_lower_bound}) and the regret approach~\citep{Orabona--Jun2021} is natural; 
we note, however, that maximizing the lower bound over $b$ in Lemma~\ref{lem:generalized_lower_bound_product} may not admit a closed-form expression for $n>1$, which necessitates to invoke a numerical optimization. We do not pursue this direction in the current paper.
Finally, \citeauthor{Orabona--Jun2021} proposed another variant PRECiSE-R70 inspired by the mixture idea in \citep{Robbins1970} and showed that the resulting confidence sequence achieves the law of iterated logarithm. }

\titrevision{At a high level, 
the lower-bound approach (LBUP) we propose in this paper is similar in spirit to \citep{Jun--Orabona2019,Orabona--Jun2021,Waudby-Smith--Ramdas2020b} that rely on a wealth lower bound.
Among these, however, LBUP is more closely related to the ``lower bound on the wealth (LBOW)'' approach of \citep{Waudby-Smith--Ramdas2020b}, in the sense that both LBUP and \citep{Waudby-Smith--Ramdas2020b} rely on a lower bound on $\log(b\frac{y}{m} + (1-b)\frac{1-y}{1-m})$; as alluded to earlier, we note that the lower bound in Lemma~16 generalizes the one used by LBOW.}
As alluded to earlier, Lemma~\ref{lem:generalized_lower_bound} is followed from a new property of the function $t\mapsto \log(1+t)$ (Lemma~\ref{lem:monotone}), which can be viewed as a higher-order extension of the proof technique of \citep[Lemma 4.1]{Fan--Grama--Liu2015}. Since this inequality has been used to derive exponential inequalities for martingales~\citep{Fan--Grama--Liu2015} and empirical Bernstein inequalities~\citep{Howard--Ramdas--McAuliffe--Sekhon2021,Waudby-Smith--Ramdas2020a}, this technique could lead to a tighter result of another problem in this domain.
The idea of the conjugate prior is also common in the literature, including the beta prior of \citet{Cover1991} and a more recent example in \citep{Howard--Ramdas--McAuliffe--Sekhon2021}.
A subtle difference in our approach is that we directly take a mixture over a lower bound with respect to a conjugate prior tailored to the new lower bound we propose, other than aiming to maximize the lower bound with a rather heuristic approximation as done in \citep{Waudby-Smith--Ramdas2020b}.

\titrevision{The topic of universal portfolio selection is a classical topic in information theory that has been studied for more than three decades, and has been established to have intimate connections to problems like data compression and sequential prediction. 
An interested reader is referred to~\cite{Cover--Thomas2006} and the references therein for a detailed treatment of these topics and their interconnections: Chapter 6 explores the horse-race formalism as well as establishes connections between betting on horse races and data compression; Chapters 11 and 13 discuss universal compression and connections to prediction and hypothesis testing; Chapter 14 investigates the closely allied concept of Kolmogorov complexity and the minimum description length (MDL) principle (see also~\cite{Grunwald2007} and~\cite{Grunwald--Roos2019} for a detailed treatise on the MDL principle); and Chapter 16 delves into portfolio selection and the UP method. Universal portfolio selection still remains a widely studied topic because a method that simultaneously achieves low regret and low time complexity has been elusive---see~\cite{Luo--Wei--Zheng2018, VanErven--VanderHoeven--Kotlowski--Koolen2020, Mhammedi--Rakhlin2022, Zimmert--Agarwal--Kale2022} and the references within for recent work in this direction.}

\newcommand{\omark}{$\mathsf{O}$}
\newcommand{\xmark}{$\mathsf{X}$}
\begin{table*}[t]
\centering

\caption{\titrevision{Comparison of existing confidence sequences based on gambling for bounded stochastic processes and proposed algorithms. 
\scndrevision{The ``convexity'' indicates if the wealth attained by a gambling strategy is guaranteed to be quasi-convex as a function of $m$, so that it naturally induces confidence \emph{intervals}; \xmark{} indicates that the corresponding wealth is not guaranteed to be quasi-convex}.
The complexity in the last column describes both (per-step) time complexity and storage complexity at time step $t$.}}
\titrevision{
\iftit
\else
\begin{small}
\fi
\setlength\tabcolsep{2pt}
\begin{tabular}{l l c c}
\toprule
\textbf{Method} & \textbf{Idea} & \textbf{Convexity} & \textbf{Complexity} \\
\midrule
GRAPA~\citep{Waudby-Smith--Ramdas2020b} & a causal approximation of the Kelly criterion~\citep{Kelly1956} & \scndrevision{\xmark}  & $O(t)$\\
LBOW~\citep{Waudby-Smith--Ramdas2020b} & GRAPA + an inequality due to \citep[Lemma~4.1]{Fan--Grama--Liu2015} & \scndrevision{\xmark} & $O(t)$\\
dKelly~\citep{Waudby-Smith--Ramdas2020b} & a mixture wealth of any finite bettors & \scndrevision{\xmark} & (bettor-dependent)\\
ConBo~\citep{Waudby-Smith--Ramdas2020b} & bet against confidence boundaries & \omark & (bettor-dependent) \\
\midrule
UP~\citep{Cover1991} & a Dirichlet mixture wealth of constant bettors & \omark & $O(t)$ \\
\midrule
PRECiSE-CO96~\citep{Orabona--Jun2021} & a lower bound on the wealth of UP via a regret upper bound & \scndrevision{\omark} & $O(t)$ \\
PRECiSE-A-CO96~\citep{Orabona--Jun2021} & PRECiSE-CO96 + \citep[Lemma~4.1]{Fan--Grama--Liu2015} & \omark & $O(1)$ \\
PRECiSE-R70~\citep{Orabona--Jun2021} & PRECiSE-CO96 with weight from \citep{Robbins1970} & \scndrevision{\omark} & $O(t)$ \\
\midrule
HR (Theorem~\ref{thm:hr_interval}) & UP restricted to two-horse race ($\equiv$ UP for $\{0,1\}^t$) & \omark & $O(1)$ \\
LBUP$(n)$ (Theorem~\ref{thm:lbup}) & a mixture of lower bounds of the wealth of constant bettors & \omark & $O(n)$ \\
HybridUP$(n)$ (Corollary~\ref{cor:hybridup}) & a combination of UP and LBUP$(n)$ & \omark & $O(n)$ \\
\bottomrule
\end{tabular}
\iftit
\else
\end{small}
\fi}
\label{tab:comparison}
\end{table*}

\section{Experiments}
\label{sec:exps}
We implemented the algorithms studied in this paper in Python and SciPy~\citep{2020SciPy-NMeth}. The code is publicly available.\footnote{\url{https://github.com/jongharyu/confidence-sequence-via-gambling}}
We translated the official MATLAB implementation\footnote{ \url{https://github.com/bremen79/precise}} of PRECiSE algorithms by the authors of \citep{Orabona--Jun2021} to Python for comparison. 

\subsection{Evolution of Wealth Processes}
\label{sec:exp:evolution}
\begin{figure*}[tp]
\centering
\begin{subfigure}[t]{\textwidth}
\centering
\includegraphics[width=\textwidth]{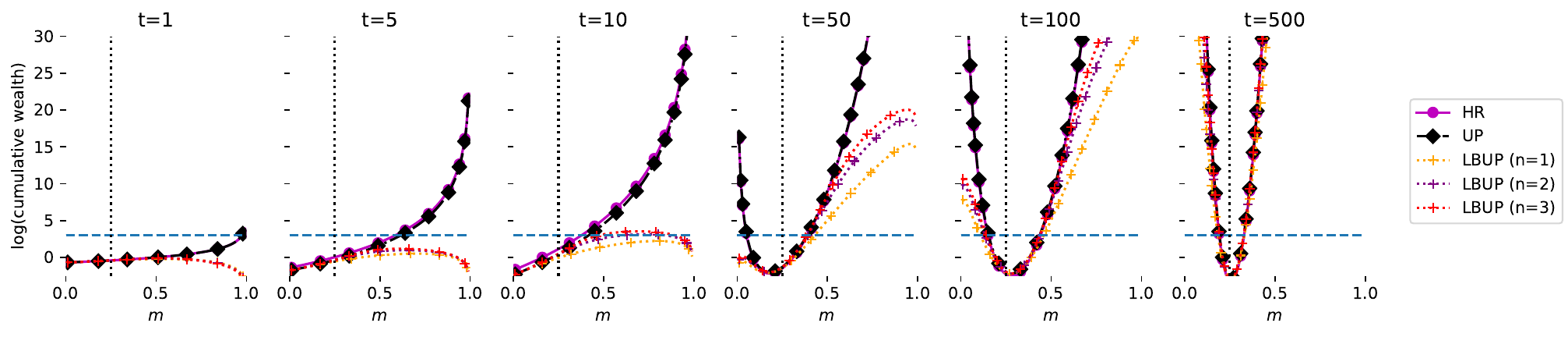}
\caption{An \iid Bern(0.25) process.}
\end{subfigure}%

\begin{subfigure}[t]{\textwidth}
\centering
\includegraphics[width=\textwidth]{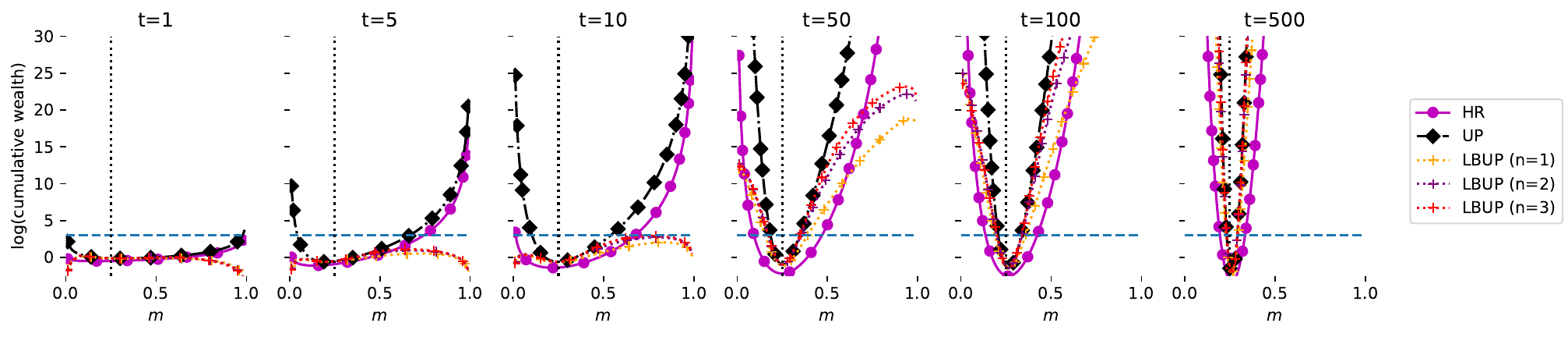}
\caption{An \iid Beta(1,3) process.}
\end{subfigure}%

\begin{subfigure}[t]{\textwidth}
\centering
\includegraphics[width=\textwidth]{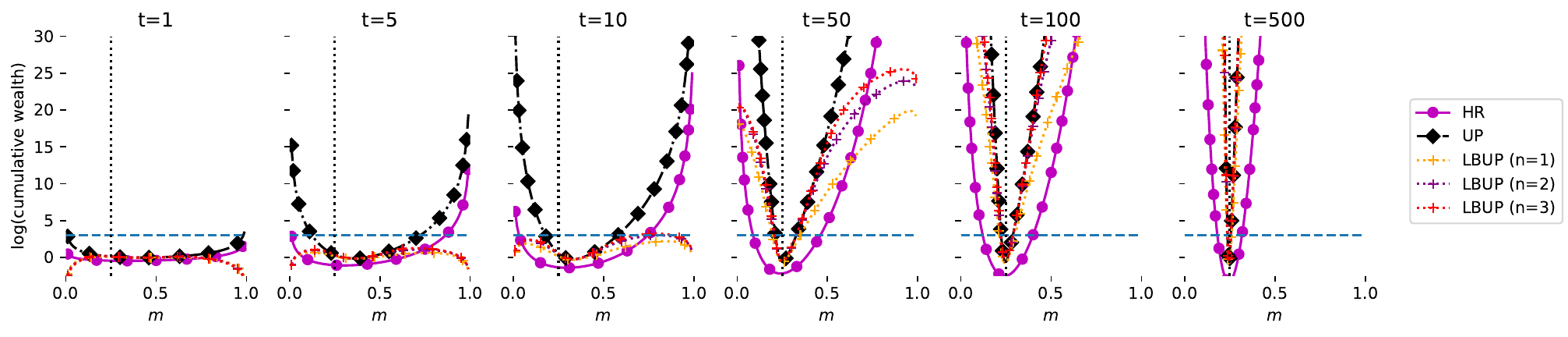}
\caption{An \iid Beta(10,30) process.}
\end{subfigure}%
\caption{The evolution of the wealth processes with respect to single realizations of \iid Bern(0.25), Beta(1,3), and Beta(10,30) processes.
Note that the true mean parameter $\mu$ is 0.25 for all three cases and is indicated by the vertical lines, while the variances are decreasing in the displayed order.
HR and UP correspond to the two-horse-race-based algorithm and the UP-based algorithm, respectively.
The $x$-axis corresponds to the parameter $m$, and the $y$-axis indicates the logarithmic cumulative wealth of each strategy for the game with a corresponding parameter $m$. The horizontal lines indicate an example threshold $\ln\frac{1}{0.05}\approx2.996$ for $\d=0.05$.}
\label{fig:evolution}
\end{figure*}

To illustrate the tightness of the proposed lower bounds, in Fig.~\ref{fig:evolution}, we visualize the evolution of the wealth processes of \titrevision{the discrete-horse-race-based gambling (Theorem~\ref{thm:hr_interval}) which we abbreviate it as HR}, UP (Theorem~\ref{thm:exact_up}, with the uniform prior), and LBUP (Corollary~\ref{cor:lbup_interval}, with the uniform prior) confidence sequences, varying the parameter $m\in(0,1)$ for $t\in\{1,5,10,50,100,500\}$.
We used single realizations of \iid processes under $\mathrm{Bern}(0.25)$, $\mathrm{Beta}(1,3)$, and $\mathrm{Beta}(10,30)$ distributions. 
Note that these distributions share the common mean $\mu=0.25$, while the variances are decreasing in the order ($\frac{3}{16}$, $\frac{3}{80}$, and $\frac{3}{656}$).

In all the cases, UP is supposed to achieve the highest wealth. For the Bern(0.25) process in (a), the HR and UP confidence sequences are equivalent, as pointed out earlier.
For the beta distributions (b) and (c), UP accumulates significantly larger wealth than HR, and the suboptimality gap gets larger as the variance decreases, which is expected by the Jensen gap in \eqnref{eq:jensen}.
We also remark that the UP and HR wealth are always log-convex, as stated in Theorems~\ref{thm:hr_interval} and \ref{thm:up_interval}.

As shown in the graph, the LBUP wealth tightly lower bounds the UP wealth.
We observe that the gap becomes smaller as the round goes on, and using a larger order $n$ consistently improves the approximation, by a significant level especially in earlier stages.
Moreover, the lower bound is particularly tight around the true mean and the typical threshold level $\ln\frac{1}{\d}<5$,\footnote{Note that $\ln\frac{1}{0.05}\approx 3$ and $\ln\frac{1}{0.01}\approx 4.6$.} and thus can be used to construct a tight outer bound of the UP confidence sequence.

\subsection{Confidence Sequences}
\begin{figure*}[tp]
\centering
\begin{subfigure}[t]{\textwidth}
\centering
\includegraphics[width=\textwidth]{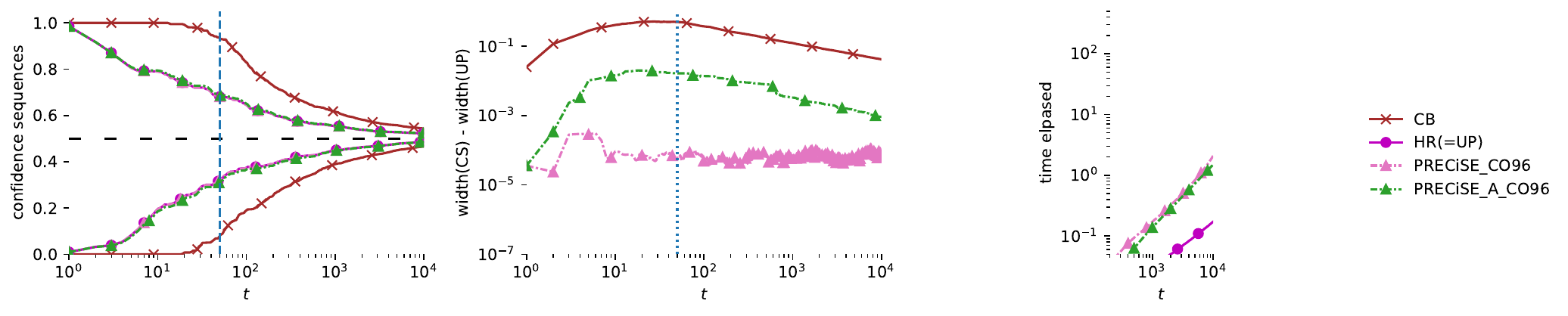}\vspace{-.5em}
\caption{With \iid $\mathrm{Bern}(0.5)$ processes.}
\end{subfigure}\vspace{.5em}%

\begin{subfigure}[t]{\textwidth}
\centering
\includegraphics[width=\textwidth]{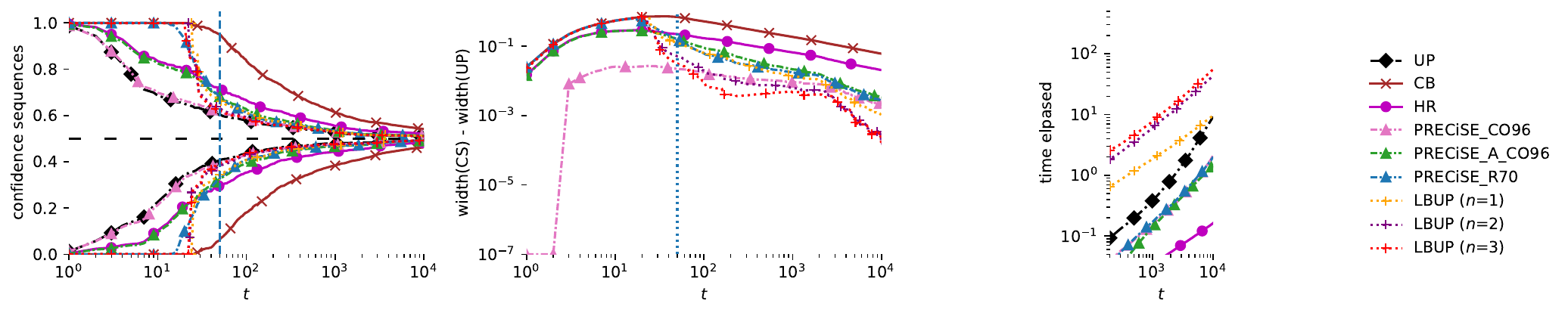}
\includegraphics[width=\textwidth]{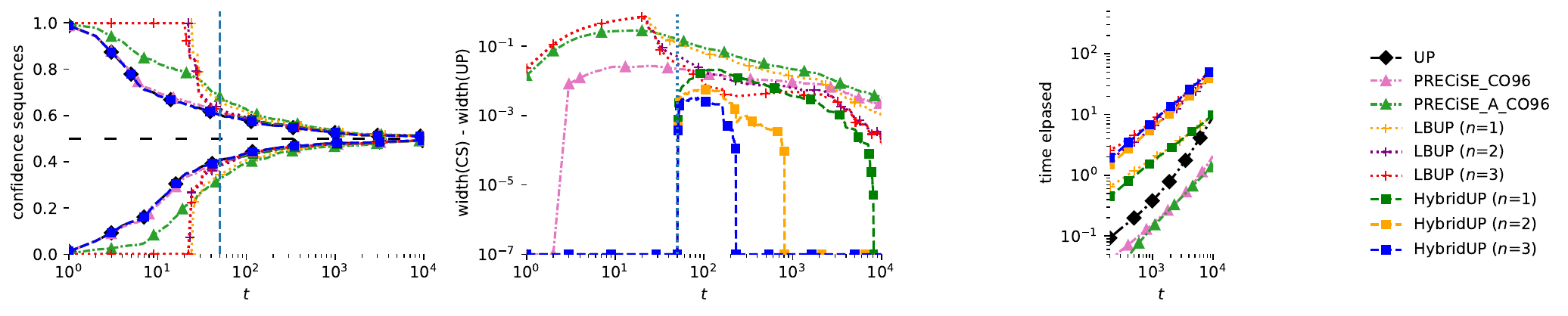}\vspace{-.5em}
\caption{With \iid $\mathrm{Beta}(1,1)$ processes.}
\end{subfigure}\vspace{.5em}%

\begin{subfigure}[t]{\textwidth}
\centering
\includegraphics[width=\textwidth]{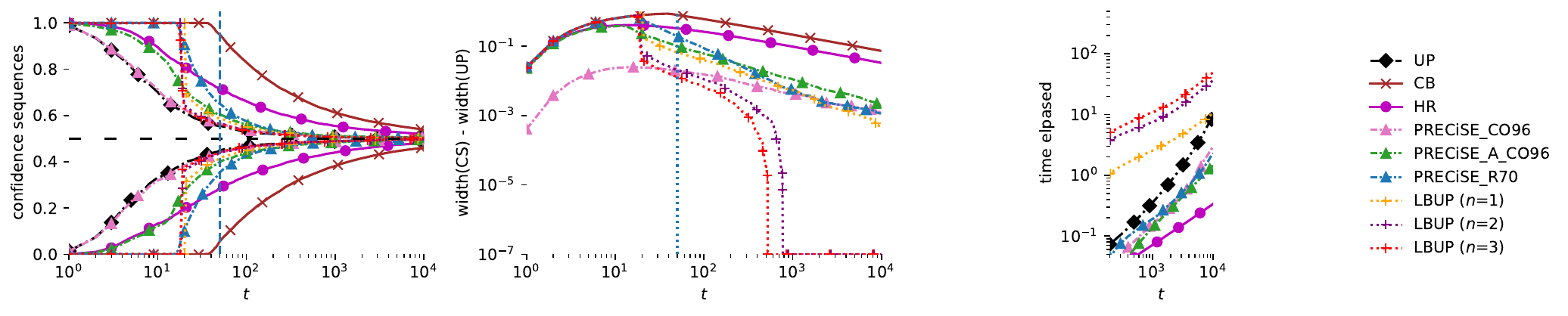}
\includegraphics[width=\textwidth]{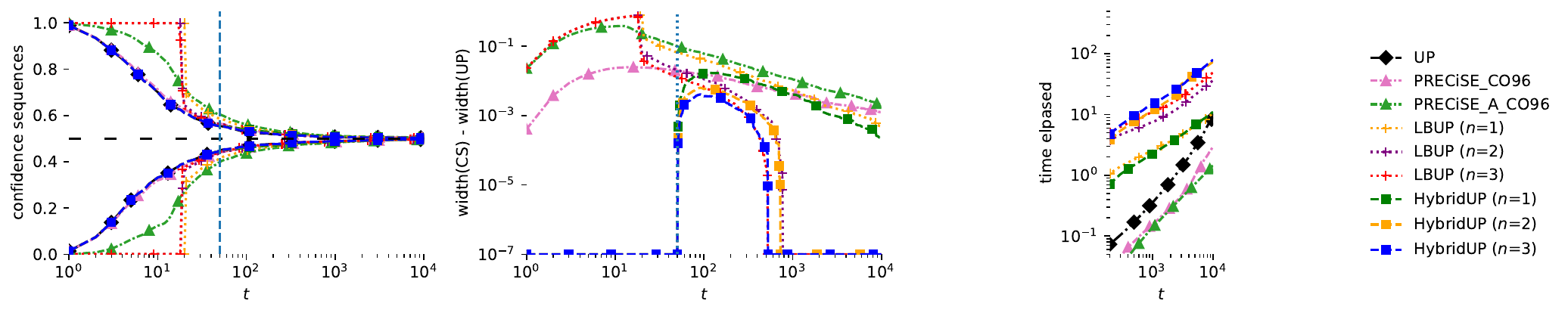}\vspace{-.5em}
\caption{With \iid $\mathrm{Beta}(10,10)$ processes.}
\end{subfigure}%

\caption{Examples of confidence sequences with respect to \iid $\mathrm{Bern}(0.5)$, $\mathrm{Beta}(1,1)$, and $\mathrm{Beta}(10,10)$ processes at level $0.95$.  
For each distribution, shown here are the averages of five independent runs. 
The first column plots the confidence sequences. 
The second column plots the gap of the size of the confidence sequences from that of UP in log scale, where we took maximums of the differences and $10^{-7}$ for visualization. (Note that for the binary process example, HR is equivalent to UP.)
The last column shows the elapsed time for each algorithm.
For the continuous process examples in (b) and (c), the results are shown in two rows to avoid clutters.
The results from UP, PRECiSE-CO96, LBUP's are shown in both rows for comparison, and the second row contains results from constant-per-step-complexity algorithms PRECISE-A-CO96 and HybridUP's.}
\label{fig:exps_0.5}
\end{figure*}

\begin{figure*}[tp]
\centering
\begin{subfigure}[t]{\textwidth}
\centering
\includegraphics[width=\textwidth]{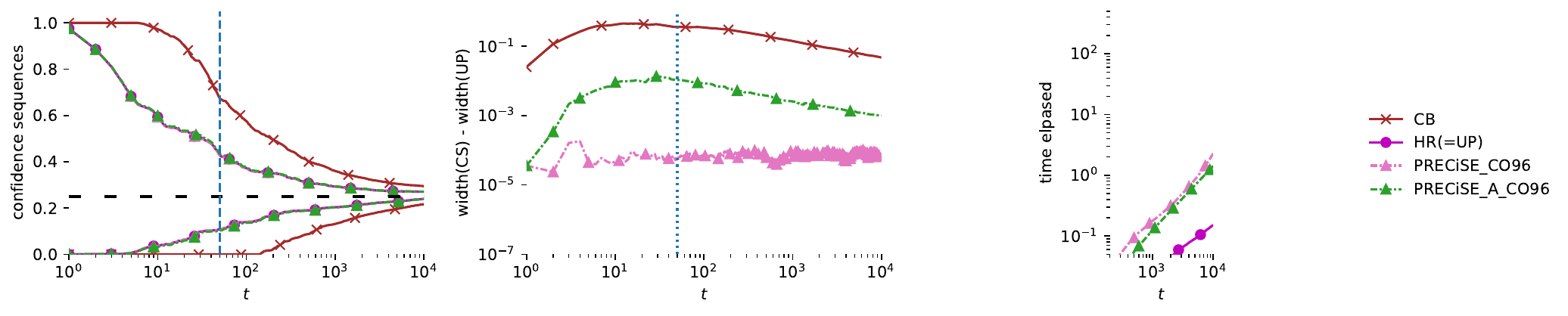}\vspace{-.5em}
\caption{With \iid $\mathrm{Bern}(0.25)$ processes.}
\end{subfigure}\vspace{.5em}%

\begin{subfigure}[t]{\textwidth}
\centering
\includegraphics[width=\textwidth]{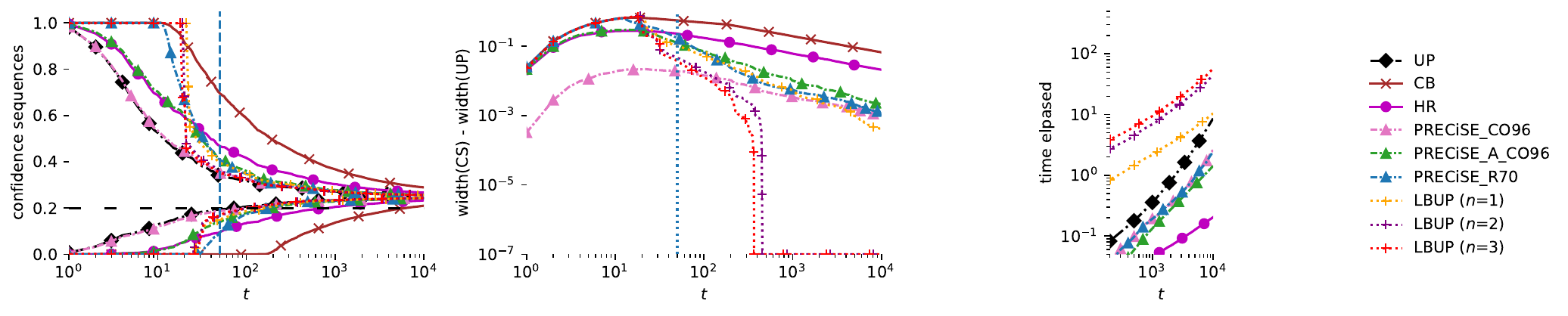}
\includegraphics[width=\textwidth]{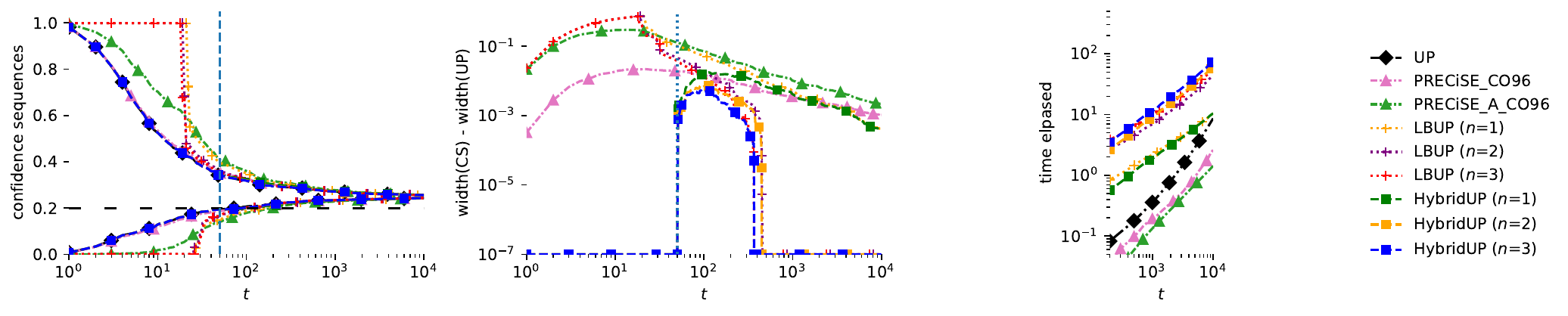}\vspace{-.5em}
\caption{With \iid $\mathrm{Beta}(1,3)$ processes.}
\end{subfigure}\vspace{.5em}%

\begin{subfigure}[t]{\textwidth}
\centering
\includegraphics[width=\textwidth]{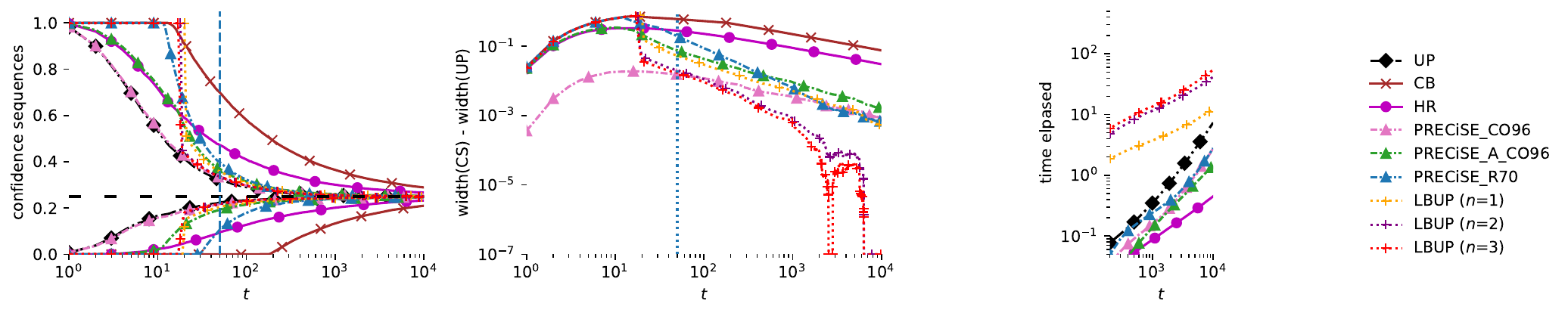}
\includegraphics[width=\textwidth]{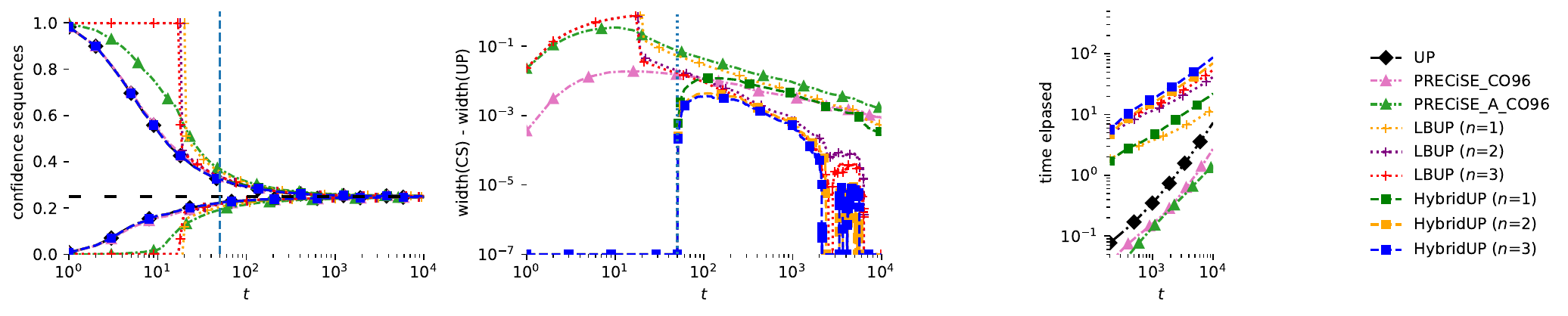}\vspace{-.5em}
\caption{With \iid $\mathrm{Beta}(10,30)$ processes.}
\end{subfigure}%

\caption{Examples of confidence sequences with respect to \iid $\mathrm{Bern}(0.25)$, $\mathrm{Beta}(1,3)$, and $\mathrm{Beta}(10,30)$ processes at level $0.95$. 
See the caption of Figure~\ref{fig:exps_0.5} for the details.}
\label{fig:exps_0.25}
\end{figure*}

We also visualized the confidence sequences of the algorithms studied in this paper and PRECiSE algorithms in Figs.~\ref{fig:exps_0.5} and \ref{fig:exps_0.25}.\footnote{We do not compare with the methods of \citet{Waudby-Smith--Ramdas2020b}, since our main focus is on developing an efficient, yet tight approximation of UP; instead, we refer an interested reader to the experiments of \citep{Orabona--Jun2021} for a comparison of \titrevision{PRECiSE algorithms} to \citep{Waudby-Smith--Ramdas2020b}.}
Similar to the experiment settings in \citep{Orabona--Jun2021}, and \citep{Waudby-Smith--Ramdas2020b} in turn, we considered \iid processes under $\mathrm{Bern}(0.5)$, $\mathrm{Beta}(1,1)$, $\mathrm{Beta}(10,10)$ (in Fig.~\ref{fig:exps_0.5}), $\mathrm{Bern}(0.25)$, $\mathrm{Beta}(1,3)$, and $\mathrm{Beta}(10,30)$ (in Fig.~\ref{fig:exps_0.25}) distributions.
Here, CB is a naive method based on an alternative odds vector construction $\xv_t\gets [1-Y_t+m, 1+Y_t-m]$, which can be viewed as running the continuous coss; cf.~\eqnref{eq:odds_vector} and see Table~\ref{tab:gambling_examples}. 
We include this simple baseline to highlight the effect of the form of gambling in constructing confidence sequences.
We used the $\mathrm{Beta}(b;\half,\half)$ prior for the UP and the UP component in the hybrid algorithms, and used $t_{\texttt{UP}}=50$ for the hybrid algorithm in~\eqnref{eq:hybrid_wealth}.

In all the cases, UP achieves the tightest confidence sequences as expected, as the others are based on lower bounds to UP.
While the HR algorithm is equivalent to UP for discrete sequences, it becomes loose as the variance decreases in (b) and (c).
Although the LBUP methods give vacuous bounds in the small sample regimes ($\sim$10 samples), but quickly approach to the behavior of UP as $t$ increases; this was also observed in Fig.~\ref{fig:evolution}.

In the second column, we plot the behavior of the confidence sequences at a finer scale, by showing the size of confidence sequences relative to that of UP.
We can see that LBUP with $n=1$ performs similar to PRECiSE at a large sample regime, and a clear improvement of LBUP with $n=2,3$ over them.
Note that LBUP with $n=2,3$ demonstrate even tighter confidence sequences compared to the PRECiSE algorithms.

Finally, we remark that the hybrid method (labeled with HybridUP) can achieve both advantages of UP and LBUP, \ie good small-sample behavior and linear time complexity, as advocated; see the second rows of (b) and (c) in Figs.~\ref{fig:exps_0.5} and \ref{fig:exps_0.25}.
The last column presents the elapsed time of each algorithm.
We remark that, \titrevision{in our experiments, CB, HR, LBUP, and HybridUP methods demonstrated $O(t^{\sim 0.8})$ time complexity, PRECiSE-A-CO96 took $O(t^{\sim 1.0})$, while UP and PRECiSE-CO96 algorithms took $O(t^{\sim 1.5})$ complexity;} 
here, the slopes of the lines correspond to the exponents in the log-log plot.

\section{Concluding Remarks}
\label{sec:conclusion}
In this paper, we established the properties of a class of universal-gambling-based algorithms for constructing confidence sequences, and also proposed a new algorithm which closely emulates the performance of the universal-portfolio-based method, with less time complexity.
Our argument is based on interpreting the existing gambling formalism as a (continuous) two-horse race.
We believe that this perspective is conceptually simpler and more intuitive, and thus allows the research community to translate and connect the results from information theory more easily.

One caveat in our experiments is that, as shown in the last columns of Figs.~\ref{fig:exps_0.5} and \ref{fig:exps_0.25}, the UP method takes less time up to about lengths $10^5$ in the current implementation, even though the LBUP and HybridUP methods incurs only linear complexity in principle. (Note also that the elapsed time for UP and PRECiSE are comparable.)
The computational bottleneck is in the numerical integration to compute the wealth lower bound in~\eqnref{eq:lbup_wealth}, as noted in Remark~\ref{rem:tradeoff}.
We expect that a specially designed numerical method for the specific integral of our interest may reduce the time complexity and lead to a faster-than-UP implementation.

An intriguing question is whether one can generalize the argument of this paper to construct a confidence sequence for high-dimensional random vectors. 
We leave this as a future research direction.


%

\appendices

\section{Technical Lemmas}

\subsection{Convexity}
\begin{lemma}\label{lem:log_convexity}
The function $m\mapsto m^{-x}(1-m)^{-(t-x)}$ is log-convex for any $t\ge 1$ and any $x\in[0,t]$.
\end{lemma}
\begin{proof}
Both $m\mapsto-\log m$ and $m\mapsto -\log(1-m)$ are convex, so is their convex combination.
\end{proof}

\begin{lemma}
\label{lem:sum_preserves_log_convexity}
If $f$ and $g$ are log-convex, then so is $f+g$.
\end{lemma}
\begin{proof}
Let $f,g\suchthat \Xc\to\Real_{>0}$, and pick any $x,y\in\Xc$.
For any $\lambda\in(0,1)$, we have
\begin{align*}
&(f(x)+g(x))^{\lambda}(f(y)+g(y))^{1-\lambda}\\
&\stackrel{(a)}{\ge} f(x)^{\lambda}f(y)^{1-\lambda} + g(x)^{\lambda}g(y)^{1-\lambda}\\
&\stackrel{(b)}{\ge} f(\lambda x+(1-\lambda)y) + g(\lambda x+(1-\lambda)y).
\end{align*}
Here, (a) follows from H\"older's inequality and (b) follows from the log-convexity of $f$ and $g$.
\end{proof}

\subsection{A New Lower Bound on a Logarithmic Function}
\label{app:log_lower_bounds}
\titrevision{
We note that \citet{Fan--Grama--Liu2015} used the following lemma as a proof technique for its Lemma~4.1.
\begin{lemma}[{\citep[eq.~4.10]{Fan--Grama--Liu2015}}]
Define $f(t)\defeq\frac{\log(1+t)-t}{t^2/2}$ for $t>-1,t\neq 0$ and $f(0)\defeq-1$. Then, $t\mapsto f(t)$ is continuous and strictly increasing over $(-1,\infty)$.
\end{lemma}
The following lemma is its higher-order generalization, which subsumes the above lemma as a special case for $\ell=2$.}
\begin{lemma}
\label{lem:monotone}
For an integer $\ell\ge 1$, if we define
\[
f_\ell(t)\defeq
\begin{cases}
\displaystyle\frac{\log(1+t)-\sum_{k=1}^{\ell-1}-\frac{(-t)^k}{k}}{\frac{(-t)^\ell}{\ell}} & \text{if $t>-1$ and $t\neq 0$,}\\
-1 &\text{if $t=0$},
\end{cases}
\]
then $t\mapsto f_\ell(t)$ is continuous and strictly increasing over $(-1,\infty)$.
\end{lemma}

Before we prove this lemma, we first compute the derivative of the function $t\mapsto f_\ell(t)$.
\begin{lemma}
\label{lem:derivative}
For any integer $\ell\ge 1$,
\[
f_\ell'(t) = -\frac{\ell}{t}\Bigl(f_\ell(t) + \frac{1}{1+t}\Bigr).
\]
\end{lemma}
\begin{proof}
From the definition, we can write
\begin{align}
\frac{f_\ell(t)}{\ell} = \frac{\log(1+t)}{(-t)^{\ell}} + \sum_{k=1}^{\ell-1}(-1)^{\ell+k} \frac{t^{k-\ell}}{k}.
\nonumber
\end{align}
Hence, as derived in Table~\ref{tab:long_eq5}, we can compute its derivative
\begin{align*}
\frac{f_\ell'(t)}{\ell} 
&= -\frac{1}{t(1+t)} - \frac{f_\ell(t)}{t},
\end{align*}
which proves the claim.
In the derivation, (a) follows from
\[
\sum_{k=1}^{\ell-1}(-t)^{k-\ell-1} = \frac{(-t)^{-\ell}(1-(-t)^{\ell-1})}{1+t}=\frac{\frac{1}{t} + \frac{1}{(-t)^{\ell}}}{1+t}.\qedhere
\]
\end{proof}

\begin{table*}
\caption{Derivation of the derivative of $\frac{f_\ell(t)}{\ell}$ in Lemma~\ref{lem:identity}.}
\centering
\begin{minipage}{\textwidth}
\begin{align*}
\frac{f_\ell'(t)}{\ell} 
&= \frac{1}{(-t)^{\ell}(1+t)} 
+\ell \frac{\log(1+t)}{(-t)^{\ell+1}} +\sum_{k=1}^{\ell-1}(-1)^{\ell+k} (k-\ell)\frac{t^{k-\ell-1}}{k}\\
&\stackrel{(a)}{=} \frac{1}{(-t)^{\ell}(1+t)} 
+\frac{\ell}{(-t)^{\ell+1}}\Bigl\{\log(1+t) +\sum_{k=1}^{\ell-1}(-1)^{k+1} \frac{t^{k}}{k}\Bigr\}
+\frac{1}{1+t}\Bigl(-\frac{1}{t}-\frac{1}{(-t)^{\ell}}\Bigr)\\
&= -\frac{1}{t(1+t)} - \frac{f_\ell(t)}{t},
\end{align*}
\medskip
\hrule
\end{minipage}
\label{tab:long_eq5}
\end{table*}

\begin{proof}[Proof of Lemma~\ref{lem:monotone}]
The continuity readily follows from applying L'H\^opital's rule, since
\[
\lim_{t\to 0} f_\ell(t) = \lim_{t\to 0} \frac{-\frac{\ell t^{\ell-1}}{1+t}}{\ell t^{\ell-1}} = \lim_{t\to 0} -\frac{1}{1+t} = -1.
\]
From Lemma~\ref{lem:derivative}, 
it is clear that proving the monotonicity is equivalent to showing that
$g_\ell(t)\le 0$ for $t\le 0$ and even $\ell$'s and $g_\ell(t)\ge 0$ otherwise, where we define
\[
g_\ell(t)\defeq -t^{\ell} \Bigl(f_\ell(t)+\frac{1}{1+t}\Bigr).
\]
Note that it is easy to show that
\[
g_\ell'(t)=\frac{t^\ell}{(1+t)^2}.
\]
Now, $g_\ell'(t)\ge 0$ for $t\le 0$ and even $n$'s, and thus $g_\ell(t)\le g_\ell(0)=0$.
For $t\ge 0$, $g_\ell'(t)\ge 0$ and thus $g_\ell(t)\ge g_\ell(0)=0$.
For $t\le 0$ and odd $n$'s, $g_\ell'(t)\le 0$ and thus $g_\ell(t)\ge g_\ell(0)=0$. This concludes the proof.
\end{proof}

\section{Deferred Proofs}
\label{app:proofs}

\subsection{Proof of Theorem~\ref{thm:equivalent_condition_achievable_wealth}}

\begin{proof}%
\titrevision{
Let $\bb_t\suchthat\Mc^{t-1}\to \Delta_{K-1}$ be any betting strategy such that 
\[
(\bb_t(\xb^{t-1}))_j \ge \frac{1}{o_j} \frac{\Psi_t(\xb^t)|_{\xb_t=o_j\eb_j}}{\Psi_{t-1}(\xb^{t-1})}
\numberthis\label{eq:wealth_to_betting}
\]
for every $t\ge 1$, every $j\in[K]$, and any $\xb^{t-1}\in\Mc^{t-1}$.
Note that, by the assumption (A1), there always exists such a betting strategy.
We prove by induction. Assume that $\wealth_{t-1}(\xb^{t-1})\ge \wealth_0 \Psi_{t-1}(\xb^{t-1})$ (induction hypothesis). Then, consider
\begin{align*}
\frac{\wealth_t}{\wealth_0}
&\stackrel{(a)}{\ge} \langle \bb_t,\xb_t\rangle \Psi_{t-1}(\xb^{t-1})\\
&\stackrel{(b)}{\ge} \sum_{j=1}^K \frac{x_{tj}}{o_j} \Psi_t(\xb^t)|_{\xb_t=o_j\eb_j} \\
&\stackrel{(c)}{\ge} \Psi_t(\xb^t).
\end{align*}
Here, (a) follows from the induction hypothesis, (b) follows from \eqnref{eq:wealth_to_betting}, and (c) follows from (A2).
}

\titrevision{For the converse, first note that
\[
\langle \bb_t,\xb_t\rangle 
=\frac{\Psi_t(\xb^{t-1}\xb_t)}{\Psi_{t-1}(\xb^{t-1})}
\numberthis\label{eq:multiplicative_game_defn_alt}
\]
for any $\xv_t\in\Mc$ by the definition of the multiplicative game.
Then, we have
\[
\frac{\sum_{j=1}^K 
\frac{1}{o_j}
\Psi_t(\xb^t)|_{\xb_t=o_j\eb_j}}{\Psi_{t-1}(\xb^{t-1})}
=\sum_{j=1}^K \frac{1}{o_j} \langle \bb_t,o_j\eb_j\rangle
=\langle \bb_t,\ones\rangle
=1,
\]
which is the condition in (A1) with equality. To verify that (A2) holds, consider
\begin{align*}
\sum_{j=1}^K \frac{x_{tj}}{o_j} \Psi_t(\xb^t)|_{\xb_t=o_j\eb_j} 
&= \Psi_{t-1}(\xb^{t-1}) \sum_{j=1}^K \frac{x_{tj}}{o_j} \frac{\Psi_t(\xb^t)|_{\xb_t=o_j\eb_j}}{\Psi_{t-1}(\xb^{t-1})} \\
&\stackrel{(d)}{=} \Psi_{t-1}(\xb^{t-1}) \sum_{j=1}^K \frac{x_{tj}}{o_j} 
\langle \bb_t, o_j\eb_j\rangle \\
&= \Psi_{t-1}(\xb^{t-1}) 
\langle \bb_t, \xb\rangle
= \Psi_t(\xb^t).
\end{align*}
Here, (d) follows from \eqnref{eq:multiplicative_game_defn_alt}. This concludes the proof.}
\end{proof}

\subsection{Proof of Corollary~\ref{cor:hr_interval}}
\begin{proof}
First, we can rewrite the equation $\tilde{\phi}_t(S_t; \frac{1}{\mu},\frac{1}{1-\mu})=\frac{1}{\d}$ as $d(\frac{S_t}{t}~\|~\mu)
=g_t(S_t;\d)$.
Since we have $d(p~\|~q)\ge 2(p-q)^2$ by Pinsker's inequality,
it readily follows that
\[
(\mu_t^{\low}(x;\d),\mu_t^{\up}(x;\d))
\subset \Bigl(\frac{S_t}{t} - \sqrt{\frac{g_t(S_t;\d)}{2}},
\frac{S_t}{t} + \sqrt{\frac{g_t(S_t;\d)}{2}} \Bigr),
\]
and the desired inequality follows from Theorem~\ref{thm:hr_interval}.
\end{proof}

\subsection{Proof of Lemma~\ref{lem:small_wealth}}
\begin{proof}
Note that, from \eqnref{eq:wealth_const_betting}, we can write
\[
\wealth_t^b\Bigl(y^t;\frac{1}{\hat{\mu}_t},\frac{1}{1-\hat{\mu}_t}\Bigr)=\sum_{k=0}^t b^k(1-b)^{t-k} \frac{y^t(k)}{\hat{\mu}_t^k(1-\hat{\mu}_t)^{t-k}}.
\]
Hence, to conclude the desired inequality, it suffices to show that the function $b\mapsto \wealth_t^b(y^t;\frac{1}{\hat{\mu}_t},\frac{1}{1-\hat{\mu}_t})$ is maximized at $b=\hat{\mu}_t$, since 
\begin{align*}
\wealth_t^{b=\hat{\mu}_t}\Bigl(y^t;\frac{1}{\hat{\mu}_t},\frac{1}{1-\hat{\mu}_t}\Bigr)
&=\sum_{k=0}^t y^t(k)\\
&= \sum_{z^t\in\{0,1\}^t} \prod_{i=1}^t y_i^{z_i}(1-y_i)^{1-z_i} \\
&= \prod_{i=1}^t (y_i + (1-y_i)) = 1.
\end{align*}
Since $b\mapsto b^k(1-b)^{t-k}$ is log-concave (Lemma~\ref{lem:log_convexity}) and a sum of any log-concave functions is also log-concave (Lemma~\ref{lem:sum_preserves_log_convexity}), so is the function $b\mapsto \wealth_t^b(y^t;\frac{1}{\hat{\mu}_t},\frac{1}{1-\hat{\mu}_t})$.
Hence, we only need to show that the derivative of the function with respect to $b$ at $b=\hat{\mu}_t$ is zero.
Indeed, we have
\begin{align*}
&\frac{\partial}{\partial b}\wealth_t^b\Bigl(y^t;\frac{1}{\hat{\mu}_t},\frac{1}{1-\hat{\mu}_t}\Bigr)\Big|_{b=\hat{\mu}_t}\\
&= \frac{\sum_{k=0}^t ky^t(k)}{\hat{\mu}_t} - \frac{\sum_{k=0}^t (t-k)y^t(k)}{1-\hat{\mu}_t}
=0,
\end{align*}
by Lemma~\ref{lem:identity} stated below.
\end{proof}

\begin{lemma}
\label{lem:identity}
For any $y^t\in\Real^t$ and $0\le k\le t$, define  
\[
y^t(k)\defeq \sum_{z^t\in\{0,1\}^t\suchthat k(z^t)=k} \prod_{i=1}^t y_i^{z_i} (1-y_i)^{1-z_i}
\]
as in \eqnref{eq:seq_k_statistics}.
Then, 
we have $\sum_{k=0}^t y^t(k)=t$ and
\[
\sum_{k=0}^t ky^t(k)=\sum_{i=1}^t y_i.
\]
\end{lemma}
\begin{proof}
Let $q_\ell(y^t)\defeq \sum_{1\le i_1<\ldots<i_\ell\le t} y_{i_1}\ldots y_{i_\ell}$ for $0\le \ell \le t$.
Then, from the recursive equation~\eqnref{eq:up_recursive_update}, it is easy to show that
\[
y^t(k)= \sum_{\ell=k}^t (-1)^{\ell+k} \binom{\ell}{k} q_\ell(y^t).
\]
With this expression, we then have
\begin{align*}
\sum_{k=0}^t ky^t(k)
&= \sum_{k=0}^t \sum_{\ell=k}^t (-1)^{\ell+k} \binom{\ell}{k} q_\ell(y^t)\\
&= \sum_{\ell=0}^t (-1)^\ell q_\ell(y^t) \sum_{k=0}^\ell (-1)^{k} k\binom{\ell}{k}\\
&\stackrel{(a)}{=} \sum_{\ell=0}^t (-1)^\ell \ell q_\ell(y^t) \sum_{k=1}^\ell (-1)^{k} \binom{\ell-1}{k-1}\\
&= q_1(y^t)=\sum_{i=1}^t y_i,
\end{align*}
which is the desired relation.
Here, (a) follows from the identity that $\sum_{j=0}^n (-1)^j \binom{n}{j}=0$ for any $n\ge 1$.
\end{proof}

\subsection{Proof of Lemma~\ref{lem:generalized_lower_bound}}
\begin{proof}
We first note that we only need to show the first case where $b\in[m,1)$ and $y\ge 0$, since the second case follows by plugging in $b\gets 1-b$, $m\gets 1-m$, $y\gets 1-y$ into the first inequality.
Now, by Lemma~\ref{lem:monotone}, we have
\begin{align*}
&\log(1+x)\\
&\ge \sum_{k=1}^{2n-1} (-1)^{k+1} \frac{x^k}{k} + \frac{x^{2n}}{x_0^{2n}} \Bigl(
\log(1+x_0) - \sum_{k=1}^{2n-1} (-1)^{k+1} \frac{x_0^k}{k}\Bigr)\\
&= \sum_{k=1}^{2n-1} \frac{(\frac{x}{x_0})^{2n} (-x_0)^k -(-x)^k}{k} + 
\Bigl(\frac{x}{x_0}\Bigr)^{2n} \log(1+x_0)
\end{align*}
for any $x\ge x_0 > -1$.
Let $g_m(b,y)=b\frac{y}{m} + (1-b)\frac{1-y}{1-m} -1$. Then, since $m\le b<1$, we have $\frac{1-b}{1-m}\le 1\le \frac{b}{m}$, and thus $g_m(b,y)\ge g_m(b,0)=\frac{1-b}{1-m}-1=\frac{m-b}{1-m}$.
The desired inequality follows by setting $x\gets g_m(b,y)$, $x_0\gets g_m(b,0)$, and observing that
\[
\Bigl(\frac{x}{x_0}\Bigr)^2
=\Bigl(\frac{g_m(b,y)}{g_m(b,0)}\Bigr)^2
=\Bigl(1-\frac{y}{m}\Bigr)^2
\]
and
\begin{align*}
-g_m(b,y)
&= 1-b\frac{y}{m}-(1-b)\frac{1-y}{1-m} \\
&= \Bigl(1-\frac{1-b}{1-m}\Bigr) \Bigl(1-\frac{y}{m}\Bigr).\qedhere
\end{align*}
\end{proof}

\section{KT Strategy for Continuous Two-Horse Race}
\label{app:kt_strategy}
Consider a gambling with odds vector $\xb_t=[o_1\ct_t,o_0(1-\ct_t)]$ for $\ct_t\in[0,1]$ given $o_1,o_0>0$.
\titrevision{Note that the following statement is stated in parallel to Theorem~\ref{thm:equivalent_condition_achievable_wealth}.}

\begin{theorem}
\label{thm:coin_betting_potential}
Let $(\psi_t\suchthat[0,t-1]\to\Real_+)$ be a sequence of nonnegative potential functions that satisfies the following condition:
\begin{itemize}
\item [(A1')] (consistency) For any $0\le x\le t-1$,
\[\psi_{t-1}(x)\ge \frac{1}{o_1} \psi_t(x+1) + \frac{1}{o_0}\psi_t(x).\]
\item [(A2')] (convexity) For any $0\le x\le t-1$,
\[\ct_t\mapsto \psi_t(x+\ct_t)\] is convex.
\end{itemize}
For each $t\ge 1$, define
\[
b_t(x) \defeq \frac{\frac{1}{o_1} \psi_t(x+1)}{\frac{1}{o_1} \psi_t(x+1) + \frac{1}{o_0} \psi_t(x)}.
\numberthis\label{eq:generalized_betting}
\]
Then, the betting strategy $b_t(\sum_{i=1}^{t-1} \ct_i)$ satisfies
\[
\wealth_t \ge \wealth_0 \psi_t\Bigl(\sum_{i=1}^t \ct_i \Bigr).
\]
\end{theorem}
\begin{proof}
We prove by induction.
Assume that
\[
\wealth_{t-1} \ge \wealth_0 \psi_{t-1}(x_{t-1}),
\]
where we denote $x_{t-1}\defeq \sum_{i=1}^{t-1} \ct_i$.
Then, we have
\begin{align*}
&\frac{\wealth_t}{\wealth_0} \\
&= (o_1\ct_t b_t + o_0(1-\ct_t)(1-b_t)) \frac{\wealth_{t-1}}{\wealth_0}\\
&\ge (o_1\ct_t b_t + o_0(1-\ct_t)(1-b_t))  \psi_{t-1}\Bigl(\sum_{i=1}^{t-1}\ct_i\Bigr)\\
&\ge (o_1\ct_t b_t + o_0(1-\ct_t)(1-b_t))  \Bigl(\frac{\psi_t(x_{t-1}+1)}{o_1} + \frac{\psi_t(x_{t-1})}{o_0}\Bigr)\\
&= \ct_t\psi_t(x_{t-1}+1) + (1-\ct_t)\psi_t(x_{t-1})\\
&\ge \psi_t(x_{t-1}+\ct_t)
=\psi_t(x_t).
\end{align*}
This concludes the proof by induction.
\end{proof}

\begin{corollary}
Recall the mixture potential defined in \eqnref{eq:defn_asymmetric_potential}:
\begin{align*}
\tilde{\phi}_t(x;o_1,o_2) \defeq o_1^{x}o_0^{t-x} \frac{B(x+\half, t-x+\half)}{B(\half,\half)}.
\end{align*}
Then, the wealth of the betting strategy
\[
\bb_t(\ct^{t-1})=\Bigl[b_t\Bigl(\sum_{i=1}^{t-1} \ct_i\Bigr),1-b_t\Bigl(\sum_{i=1}^{t-1} \ct_i\Bigr)\Bigr]
\]
with
\begin{align*}
b_t(x) \defeq \frac{1}{t}\Bigl(x + \half\Bigr)
\end{align*}
for $x=x_{t-1}\in [0,t-1]$, 
is lower bounded by 
\[
\wealth_0\tilde{\phi}_t\Bigl(\sum_{i=1}^{t-1}\ct_i;o_1,o_2\Bigr).
\]
\end{corollary}
\begin{proof}
It is easy to check that the potential satisfies both conditions (A1') and (A2').
Indeed, (A1') holds with equality, and (A2') follows as a corollary by the logarithmic convexity of the mapping $x\mapsto \psi_t(x)$ over $[0,t]$.
Therefore, the induced betting strategy defined in \eqnref{eq:generalized_betting} achieves the wealth at least the coin betting potential $\psi_t(x_t)$.
Here, note that, for any $o_1,o_0>0$, 
\begin{align*}
b_t(x) &= \frac{\frac{1}{o_1} \psi_t(x+1)}{\frac{1}{o_1} \psi_t(x+1) + \frac{1}{o_0} \psi_t(x)}
= \frac{1}{t}\Bigl(x + \half\Bigr)
\end{align*}
for $x=x_{t-1}\in [0,t-1]$, which is equivalent to the KT strategy for the standard even-odds case $o_1=o_0$.
Invoking Theorem~\ref{thm:coin_betting_potential} concludes the proof.
\end{proof}

\section*{Acknowledgment}
The authors appreciate F. Orabona, K.-S. Jun, and A. Ramdas for providing constructive comments on an earlier version of this manuscript.
The authors also are grateful to Y.-H. Kim for his comments on an earlier draft and his support on this research.


\ifCLASSOPTIONcaptionsoff
  \newpage
\fi



%
\bibliographystyle{plainnat}
\bibliography{ref}

%



\newpage
\begin{IEEEbiographynophoto}{J. Jon Ryu} is currently a postdoctoral associate at the Department of Electrical Engineering and Computer Science at Massachusetts Institute of Technology.
He received the B.S. (Hons.) degrees in electrical and computer engineering and mathematical science (double major) from Seoul National University, Seoul, South Korea, in 2015, and the
Ph.D. degree in electrical engineering from the University of California San
Diego in 2022. 
He was a recipient of Kwanjeong Scholarship for graduate study from 2015 to 2020.
His research interests focus on developing scalable and reliable machine learning algorithms based on the first principles of information theory and statistics.
\end{IEEEbiographynophoto}

\begin{IEEEbiographynophoto}{Alankrita Bhatt} received the bachelor’s degree in electrical engineering from the Indian Institute of Technology, Kanpur, India, in 2016 and the Ph.D. degree in electrical and computer engineering from the University of California, San Diego, in 2022. She was previously a postdoc at Simons Institute for the Theory of Computing, UC Berkeley and is currently a postdoc in the CMS department at Caltech. 
\end{IEEEbiographynophoto}
\vfill



\end{document}